\documentclass[a4paper]{siamart190516}

\usepackage{amsfonts, amssymb, bm}
\usepackage{graphicx}
\usepackage{verbatim}
\usepackage{algorithm}
\usepackage{algpseudocode}
\usepackage{longtable}
\usepackage{multirow}
\usepackage{diagbox}
\usepackage{lscape}
\usepackage{color}

\newsiamremark{remark}{Remark}
\newsiamremark{assumption}{Assumption}

\headers{An Inexact ALM for SOCPs with Applications}{L. Liang, D. F. Sun, and K.-C. Toh}

\title{An Inexact Augmented Lagrangian Method for Second-order Cone Programming with Applications}

\author{
Ling Liang~\thanks{Department of Mathematics, National University of Singapore, 10 Lower Kent Ridge Road, Singapore 119076 \mbox{(\email{liang.ling@u.nus.edu})}.}
\and 
Defeng Sun~\thanks{Department of Applied Mathematics, The Hong Kong Polytechnic University, Hung Hom, Hong Kong \mbox{(\email{defeng.sun@polyu.edu.hk})}. The research of this author is supported in part by Hong Kong Research Grant Council under grant PolyU 153014/18P.}
\and 
Kim-Chuan Toh~\thanks{Department of Mathematics, and Institute of Operations Research and Analytics, National University of Singapore, 10 Lower Kent Ridge Road, Singapore 119076 \mbox{(\email{mattohkc@nus.edu.sg})}. This author is supported in part by the Ministry of Education, Singapore, under its Academic Research Fund Tier 3 grant (MOE-2019-T3-1-010).}
}

\usepackage{amsopn}

\newcommand{\mc}{\multicolumn}

\def\norm#1{\left\lVert #1\right\rVert}
\def\inprod#1#2{\big\langle #1,\,#2\big\rangle}

\newcommand{\cT}{{\cal T}}

\newcommand{\cM}{{\cal M}}
\newcommand{\cN}{{\cal N}}

\newcommand{\cC}{{\cal C}}

\newcommand{\cB}{{\cal B}}
\newcommand{\bea}{\begin{eqnarray*}}
\newcommand{\eea}{\end{eqnarray*}}

\def\mc{\multicolumn}

\def\inprod#1#2{\langle#1,\,#2\rangle}

\def\grad{\nabla}

\def\S{{\cal S}}

\def\cK{{\cal K}}
\def\cB{{\cal B}}

\def\S{\mathbb{S}}

\def\R{\mathbb{R}}

\def\S{\mathbb{S}}

\def\X{\mathbb{X}}
\def\Y{\mathbb{Y}}
\def\Z{\mathbb{Z}}
\def\W{\mathbb{W}}

\ifpdf
\hypersetup{
  pdftitle={An Inexact ALM for SOCPs with Applications},
  pdfauthor={L. Liang, D. F. Sun, and K.-C. Toh}
}
\fi

\begin{document}

\maketitle

\begin{abstract}
	In this paper, we adopt the augmented Lagrangian method (ALM) to solve convex quadratic second-order cone programming problems (SOCPs). Fruitful results on the efficiency of the ALM have been established in the literature. Recently, it has been shown in [Cui, Sun, and Toh, {\em Math. Program.}, 178 (2019), pp. 381--415] that if the quadratic growth condition holds at an optimal solution for the dual problem, then the KKT residual converges to zero R-superlinearly when the ALM is applied to the primal problem. Moreover, Cui, Ding, and Zhao [{\em SIAM J. Optim.}, 27 (2017), pp. 2332-2355] provided sufficient conditions for the quadratic growth condition to hold under the metric subregularity and  bounded linear regularity conditions for solving composite matrix optimization problems involving spectral functions. Here, we adopt these recent ideas to analyze the convergence properties of the ALM	when applied to SOCPs. To the best of our knowledge, no similar work has been done for SOCPs so far. In our paper, we first provide sufficient conditions to ensure the quadratic growth condition for SOCPs. With these elegant theoretical guarantees, we then design an SOCP solver and apply it to solve various classes of SOCPs, such as minimal enclosing ball problems, classical trust-region subproblems, square-root Lasso problems, and DIMACS Challenge problems. Numerical results show that the proposed ALM based solver is efficient and robust compared to the existing highly developed solvers, such as Mosek and SDPT3.
\end{abstract}



\section{Introduction}
\label{sec-introduction}
Denote the standard $d$-dimensional second-order cone (also called ice cream cone or Lorentz cone) in $\R^d$ $(d\ge 1)$ as
\begin{eqnarray*}
	\cK^d:=
	\left\{
	x = (x_0,x_t)^\top \in \R\times \R^{d-1} \;\left\vert\; x_0\geq \norm{x_t} \right.
	\right\}.
\end{eqnarray*}
Let $\cK$ be the Cartesian product of $r$ second-order cones, i.e.,
\begin{eqnarray*}
	\cK = \cK^{n_1}\times \cdots \times \cK^{n_r}\subseteq \R^n,
\end{eqnarray*}
where $n = n_1+\cdots+n_r$.
In this paper, we consider the following convex quadratic second-order cone programs (SOCPs)
\begin{eqnarray*}
	{\rm (P)}\quad \; &\min_{x=(x_1;x_2;x_3)}&\; f^0(x):=\frac{1}{2}\inprod{x_1}{Hx_1}-\inprod{b}{x_2}+\delta_{\cK}(x_3) \\[5pt]
	&{\rm s.t.} &\; -Hx_1+A^\top x_2+x_3 = c,\; x_1\in {\rm Ran}\,(H)\subseteq \R^n,\;
	x_2\in \R^m,\; x_3\in \R^n,
\end{eqnarray*}
where $H\in \S_+^n$ (the cone of $n\times n$ symmetric positive semidefinite matrices) and $A\in\R^{m\times n}$ are given matrices, ${\rm Ran}(H)$ denotes the range space of $H$, $c\in \R^n$ and $b\in \R^m$ are given vectors, and $\delta_{\cK}(\cdot)$ is the indicator function for the symmetric cone $\cK$. In the above, $(x_1;x_2;x_3)$ denotes the concatenation of the vectors $x_1,x_2,x_3$. For notational simplicity, we denote $\X :={\rm Ran}(H)\times \R^m \times \R^n$ for the rest of this paper. The dual problem associated with ${\rm (P)}$ is given by
\begin{eqnarray*}
{\rm (D)}\quad \; & \max_{y} &\; g^0(y):=-\frac{1}{2}\inprod{y}{Hy}-\inprod{c}{y} - \delta_{\cK}(y) \\[5pt]
&{\rm s.t.} &\; Ay = b,\; y\in \mathbb{R}^{n}.
\end{eqnarray*}
We should mention that in this paper, our naming convention of the primal and dual problems is opposite of the convention adopted in the interior-point methods (IPMs) literature.

Let ${\rm SOL_P}$ and ${\rm SOL_D}$ be the solution sets of ${\rm (P)}$ and ${\rm (D)}$, respectively. The KKT optimality condition for ${\rm (P)}$ and ${\rm (D)}$ is given as follows:
\begin{equation}
\label{kkt-condition}
-Hx_1+A^\top x_2+x_3 = c,\quad Ay -b = 0,\quad H(x_1-y) = 0, \quad \cK \ni x_3 \perp y\in \cK.
\end{equation}
We assume for the rest of this paper that the KKT condition~\eqref{kkt-condition} admits at least one solution. Under this assumption, it is well known that $(\bar{x},\bar{y})$ solves the KKT condition \eqref{kkt-condition} if and only if $\bar{x}\in {\rm SOL_P}$ and $\bar{y}\in {\rm SOL_D}$.

Note that problems ${\rm (P)}$ and ${\rm (D)}$ cover the standard primal and dual linear SOCP problems by simply dropping the quadratic term in the objective function, respectively. One may also observe that problem ${\rm (P)}$ or ${\rm (D)}$ can be reformulated as a linear SOCP with additional affine and rotated quadratic cone constraints. To explain the procedure, we consider problem ${\rm (D)}$ as an illustrative example. Recall that a $d$-dimensional $(d\geq 3)$ rotated quadratic cone is defined by
\begin{eqnarray*}
	\cK_r^d:=
	\big\{
	x = (x_1,x_2,\cdots,x_d)\in \R^d\;|\; 2x_1x_2\geq x_3^2 + \cdots +x_d^2,\; x_1,x_2\geq 0
	\big\}.	
\end{eqnarray*}
From the positive semidefiniteness of $H$, there exists $ R\in \mathbb{R}^{k\times n} $ with $k := \mathrm{rank}(H)\leq n$, such that $H = R^\top R$ and hence we can rewrite problem ${\rm (D)}$ as
\[
\min_{y,t}
\big\{
t+ \inprod{c}{y} \;\left\vert\;  Ay = b,\; y\in \cK,\; \norm{Ry}^2\leq 2t
\right.
\big\}.
\]
Observe that the constraint $\norm{Ry}^2\leq 2t $ is equivalent to $(t,1,Ry)\in \cK_r^{k+2}$. Therefore,  (D) can be reformulated as
\begin{eqnarray}
\label{eq-reformulatedsocp}
\min_{y,t,s,z}
\left\{
t+ \inprod{c}{y} \;\left\vert\;  Ay = b,\; Ry - z = 0,\;s = 1,\; y\in \cK,\; (t,s,z)\in \cK_r^{k+2}
\right.
\right\}.
\end{eqnarray}
From the constraints in \eqref{eq-reformulatedsocp}, we can infer the following potential disadvantages for transforming the quadratic term in the objective into the constraints: (1) One needs to introduce an affine constraint with coefficient matrix of size $(k+1)\times (n+k+2)$. Thus, when $k$ is large, this additional affine constraint will increase the difficulty of computing the search direction (e.g., when an IPM is used, one needs to solve a large linear system to compute the Newton direction). (2) Introducing the extra variables $(y,z,s)$ naturally would increase the computational complexity in solving the problem. (3) The factorization $H=R^\top R$ to begin with can be expensive to compute. The above disadvantages have motivated us to deal with ${\rm (P)}$ and ${\rm (D)}$ directly.

Optimization problems with second-order cone constraints have been studied for quite a long time and still receive constant attention to date. There is a large body of literature on the topic. For comprehensive surveys and numerous important applications of  SOCPs, we refer the reader to~\cite{alizadehsocp2003,KuoORapp2004,LoboAppsocp1998} and references therein. Here, we mention some recent literature in the next three paragraphs to capture the main research topics  on SOCPs.

Optimization problems with second-order cone constraints are of great interest theoretically due to their nonpolyhedral nature. In fact, theoretical results on variational analysis for SOCPs have been well developed. For example, Bonnans and Ram\'{i}rez C.~\cite{JFBonnans2005} performed rigorous and systematic perturbation analysis for nonlinear SOCPs. Outrata and Sun~\cite{OutrataCoderivative2008} then computed the limiting (Mordukhovich) coderivative of the metric projection onto a second-order cone, which can be used to provide a sufficient condition for the Aubin property of the solution map of a complementarity problem as well as to derive certain necessary optimality conditions. Very recently, Hang Mordukhovich, and Sarabi \cite{Hangvariational2018} conducted a second-order variational analysis for SOCPs without imposing any nondegeneracy assumptions.

The importance of SOCPs comes from their modeling power.
Indeed, applications of SOCPs have grown dramatically over the years in engineering, control, management science, and statistics; see for instance~\cite{Baradarpowerflow2013,Bellonisrlasso2011,GoldfarbTV2005,Makrodupperlimitana2007,NamTRS2017,Shivaswamymissing2006,Tsengsensor2007,ZhouMEB2005}. As illustrative examples, we consider minimal enclosing ball problems~\cite{ZhouMEB2005}, classical trust-region subproblems~\cite{NamTRS2017}   and  square-root Lasso problems~\cite{Bellonisrlasso2011} in this paper.

As driven by the needs in applications, many algorithms have also been developed for solving SOCPs. Among them, the most well-developed ones are IPMs. In particular, primal-dual IPMs have been shown to have superior theoretical and practical efficiency, and they are widely used to solve SOCPs to high precision. For references on primal-dual IPMs for solving SOCPs, we recommend~\cite{AndersenIPM2003,CaiAUG2006,Monteirosocp2000,NemirovskiQCQP1996,NesterovIPM1994,Tsuchiyasocp1999}. However, IPMs are sometimes not scalable for large-scale problems due to the high expense needed to solve the large linear system of equations in each iteration. Besides IPMs, smoothing Newton methods~\cite{Chensmoothingnewton2003,Fukushimasmoothing2001} and  semismooth Newton methods~\cite{Kanzowsemismoothsocp2006} have also been applied to solve the KKT system directly. However, limited numerical implementations and experiments were conducted in these works. Therefore, the practical performance of these algorithms remains unclear. Finally, the augmented Lagrangian method (ALM) has also been applied to general nonlinear programming problems with the second-order cone constraint in~\cite{Hangvariational2020,liu2007}. Both papers focus on analyzing the local fast convergence rate of the ALM under some strong conditions, such as the uniform second-order growth condition and the second-order sufficient condition, but with different approaches. Nevertheless, the practical performance of the ALM is not considered in both works. Therefore, the contributions in~\cite{Hangvariational2020,liu2007} are mainly on the theoretical development.

Continuing the research theme on algorithmic development just mentioned above, the present paper aims to design a highly efficient and scalable algorithm for solving large-scale SOCPs. Our algorithmic design is motivated by the recent success in developing an ALM framework for solving semidefinite programming (SDP) problems. Specifically, in \cite{ZhaoSTNewtonCGALM2010}, an inexact ALM combined with a semismooth Newton method has been shown to be highly efficient and scalable for solving large-scale SDP problems. Thus, it is natural for us to apply a similar ALM framework to solve SOCPs directly. Note that this ALM framework, together with its convergence analysis, is well established based on the theoretical work of Rockafellar~\cite{rockafellar1976augmented,rockafellar1976ppa}. Along this line, various papers (see, e.g.,~\cite{CST2017,luqueasymptotic1984}) have extended Rockafellar's work by relaxing some restrictive conditions for convergence. For instance, Cui, Sun, and Toh \cite{CST2017} showed recently that under the calmness condition for the dual solution mapping (equivalently, the quadratic growth condition for the dual problem), the ALM applied to a primal convex composite conic programming problem has an asymptotic R-superlinear convergence rate in term of the KKT residual. Moreover, Cui, Ding, and Zhao \cite{CDZhao2016} showed that under the metric subregularity and bounded linear regularity conditions, the quadratic growth condition can be guaranteed for matrix optimization problems involving symmetric spectral functions. Therefore, we can borrow these ideas to establish the fast convergence rate of the ALM when applied to SOCPs. To the best of our knowledge, no such work has been done for SOCPs so far.

Our contributions in this paper can thus be summarized as follows:
\begin{itemize}
	\item Theoretically, we provide sufficient conditions for ensuring the quadratic growth condition for the dual problem ${\rm (D)}$ under the bounded linear regularity condition and the metric subregularity condition. In particular, we revisit the fact that if a strictly complementary solution exists, then the quadratic growth condition holds for problem ${\rm (D)}$. Thus, sufficient conditions for the R-superlinear convergence of the KKT residual generated by the ALM can also be obtained.

	\item Numerically, we develop a highly efficient and robust SOCP solver for  large-scale SOCPs. Our numerical results show that the solver is comparable to existing state-of-the-art linear SOCP solvers, such as the highly powerful commercial solver Mosek and the efficient open source solver SDPT3, when solving some large-scale linear SOCPs. More specifically, we apply our SOCP solver to solve minimal enclosing ball (MEB) problems, square-root Lasso problems, and some linear SOCPs in DIMACS challenge. For the SOCPs arising from the MEB problems, we show that any feasible solution to the primal problem is constraint nondegenerate and hence the semismooth Newton method employed to solve the ALM subproblems is guaranteed to attain at least a superlinear convergence rate.
	
	\item For solving the convex quadratic SOCPs ${\rm (P)}$ and ${\rm (D)}$, we deal with the quadratic objective functions directly in a concise manner. We do not need to transform the problem into a much larger linear SOCP problem with an additional rotated quadratic cone constraint. The great computational benefit of our approach is demonstrated via the numerical results for solving the classical trust-region subproblems.
\end{itemize}

The rest of the paper is organized as follows. In section~\ref{sec-pre}, we introduce some preliminaries and notation which will be used in this paper. Recently developed convergence results of the ALM and related topics on the quadratic growth condition for the dual problem ${\rm (D)}$ are presented in sections \ref{sec-ALM} and \ref{sec-quad-growth}. A highly efficient semismooth Newton method for solving the ALM subproblems is presented in section~\ref{sec-issn} with some well-known convergence properties. In section~\ref{sec-numexp}, we design an SOCP solver based on the proposed ALM. Moreover, we discuss the efficient implementation of the solver and conduct extensive numerical experiments to illustrate the efficiency and robustness of the proposed algorithm. Finally, we conclude the paper in section~\ref{sec-conclusion}.

\section{Preliminaries}
\label{sec-pre}
In this section, we first list some notation and present some basic
material on the projection operator onto the standard second-order cone.

\subsection{Notation and definitions}
\label{subsec-notation}
We use $\Y$, $\Z$ and $\W$ to denote generic finite-dimensional real Euclidean spaces. For a given closed convex cone $\cC$, we use $\cC^\circ$ and $\cC^*$ to denote the polar and dual cones of $\cC$, respectively. We use $N_\cC(x)$ and $\cT_\cC(x)$ to denote the normal and tangent cones of $\cC$ at a point $x\in \cC$, respectively.	

Let $ f:\W \rightarrow [-\infty, +\infty] $ be a given convex function. The effective domain of $f$ is denoted as ${\rm dom}\,(f)$. Moreover, the subdifferential of $f$ at the point $x\in {\rm dom}(\, f)$ is denoted as $\partial f(x)$. We use $f^*$ to denote the convex conjugate function of $f$, i.e., $f^*(z) = \sup_x\,\{\inprod{z}{x}-f(x)\;|\; x\in {\rm dom}\,(f)\}$. Let $D\subseteq \W$ be a set. We use $\delta_D(\cdot)$ to denote the indicator function over the set $D$. If the set $D$ is closed and convex, then the metric projection of $x\in \W$ onto $D$ is defined by $\Pi_D(x):={\rm arg}\min\{ \norm{x-s}\;|\;s\in D \}$. Moreover, the distance for a point $x\in \W$ to the set $D$ is given by ${\rm dist}\,(x,D):=\inf_{x\in D}\,\norm{x-d}$. For more useful properties related to convex functions and convex sets, we refer the reader to the monograph of Rockafellar~\cite{rockafellar1970convex}.

The following definitions on the Lipschitz-like continuity for a set-valued mapping are commonly involved in derivation of the convergence rate for the ALM.

\begin{definition}\label{def-Lip}
	\begin{enumerate}
		\item A set-valued mapping $\Phi:\W\rightrightarrows\Y$ is Lipschitz continuous at $u\in \W$ with modulus $\kappa > 0$ if $\Phi(u) = \{v\}$ and there exists a positive constant $\epsilon$ such that
		\begin{eqnarray*}
			\norm{v'-v}\leq \kappa \norm{u'-u}\quad \forall\,v'\in \Phi(u'),\quad u'\in \mathbb{B}_\epsilon(u).
		\end{eqnarray*}
		\item A set-valued mapping $\Phi:\W\rightrightarrows\Y$ is upper Lipschitz continuous at $u\in \W$ with modulus $\kappa > 0$ if there exists a positive constant $\epsilon$ such that
		\begin{eqnarray*}
			{\rm dist}\,(v',\Phi(u))\leq \kappa \norm{u'-u}\quad \forall\,v'\in \Phi(u'),\quad u'\in \mathbb{B}_\epsilon(u).
		\end{eqnarray*}
	\end{enumerate}
\end{definition}

Next, we define some mappings that are closely related to the perturbation theory of optimization problems. We will use these mappings to analyze the convergence property of the proposed ALM.

Let $l:\X\times \R^n\rightarrow [-\infty,+\infty]$ be the Lagrangian function in the extended form:
\begin{eqnarray*}
l(x,y):=
\begin{cases}
f^0(x)+\inprod{y}{-Hx_1+A^\top x_2+x_3-c} & x\in \text{dom}(f^0),\\
+\infty & x\notin \text{dom}(f^0).
\end{cases}
\end{eqnarray*}
Denote the essential objective functions of $ {\rm (P)} $ and $ {\rm (D)} $, respectively, by
\begin{eqnarray*}
	f(x) & :=&\sup_y\;l(x,y) =
	\begin{cases}
		f^0(x)  & -Hx_1+A^\top x_2+x_3=c,\\
		+\infty & {\rm otherwise},
	\end{cases}
	\\[5pt]
	g(x) &:=& \inf_x\;l(x,y)=
	\begin{cases}
		g^0(y)  & Ay = b,\\
		-\infty & {\rm otherwise}.
	\end{cases}
\end{eqnarray*}
Note that the functions $l(\cdot)$, $f(\cdot)$ and $g(\cdot)$ are convex-concave, convex and concave, respectively. Therefore, their subdifferentials are well-defined. In particular, we can define the following set-valued mappings $T_l:\X\times \R^n \rightrightarrows\X\times \R^n$, $T_f:\X\rightrightarrows \X$, and $T_g:\R^n\rightrightarrows \R^n$ by
\begin{equation*}
T_l(x,y):=\big\{ (u,v)\in \X\times \R^n\;\left\vert\; (u,-v)\in \partial l(x,y) \right.\big\},\quad (x,y)\in \X\times\R^n,
\end{equation*}
$ T_f := \partial f$, and $ T_g:=-\partial g $, respectively.

Consider the following linearly perturbed form of problem $ {\rm (P)} $ with perturbation parameters $(u,v)\in \X\times \R^n$:
\begin{eqnarray*}
({\rm P}(u,v))\quad \min_x\;
\left\{
f^0(x)-\inprod{x}{u}\;\left\vert\; -Hx_1+A^\top x_2+x_3+v-c = 0
\right.
\right\}.
\end{eqnarray*}
Then according to~\cite{rockafellar1976augmented}, the inverse mapping of three mappings $T_l$, $T_f$, and $ T_g $ are well-defined (since $T_l$, $T_f$, and $ T_g $ are shown to be maximal monotone operators) and can be viewed as the solution mappings of their corresponding perturbed problems. Indeed, one can verify that
\begin{eqnarray*}
	\left\{
	\begin{array}{rll}
		T_l(u,v)^{-1} &=\; \text{the set of all KKT points to } ({\rm P}(u,v)),\\[5pt]
		T_f(u)^{-1}  &= \;\text{the set of all optimal solution to }({\rm P}(u,0)),\\[5pt]
		T_g(v)^{-1}  &= \; \text{the set of all optimal solution to }({\rm D}(0,v)),
	\end{array}
	\right.
\end{eqnarray*}
where $({\rm D}(u,v))$ is the ordinary dual of $ ({\rm P}(u,v)) $ for any $(u,v)\in \X \times \R^n$. Therefore, we may call $T_l^{-1}$ the KKT solution mapping, $T_f^{-1}$ the primal solution mapping, and $T_g^{-1}$ the dual solution mapping.

\subsection{Projection onto the second-order cone}
\label{subsec-projsoc}
We next recall some important properties on the projection onto the second-order cone. We will pay particular attention to the differential properties for the projection mapping $\Pi_K(\cdot)$, where for notational simplicity  we use $K$ to denote  a single second-order cone in $\R^d$, i.e.,
\begin{eqnarray*}
	K:=\{ x=(x_0,x_t)^\top\in \R^d \;|\; x_0\geq \norm{x_t}\}.
\end{eqnarray*}
The following lemma provides an exact formula of the projection onto the second-order cone (see, e.g.,~\cite{Fukushimasmoothing2001}).
\begin{lemma}
	\label{lem-projsoc}
	For any $x=(x_0,x_t)^\top\in \R^d$, the projection onto the second-order cone $K$ is given by
	\begin{eqnarray*}
		\Pi_K(x) =
		\begin{cases}
			x & \norm{x_t}\leq x_0,\\
			0 & \norm{x_t}\leq -x_0,\\
			\frac{1}{2}(x_0+\norm{x_t})  {\Big(1,\frac{x_t}{\norm{x_t} }\Big )^{\top}} &{\rm otherwise}.
		\end{cases}
	\end{eqnarray*}
\end{lemma}
Since $\Pi_K(\cdot)$ is a globally Lipschitz continuous mapping with modulus 1 on $\R^d$, i.e.,
\begin{eqnarray*}
	\norm{\Pi_K(x)-\Pi_K(y)}\leq \norm{x-y} \quad \forall\, x, \,y \in \R^d,
\end{eqnarray*}
it is well known that by Rademacher's Theorem~\cite{FedererGeometric2014}, $\Pi_K(\cdot)$ is Fr\'{e}chet differentiable almost everywhere on any open set $\mathcal{O} \subseteq \R^d$. Thus, we can define the B-subdifferential of $\Pi_K(\cdot)$ at a point $x\in \R^d$ as
\begin{eqnarray*}
	\partial_B \Pi_K(x):=\{ \lim_{i\rightarrow \infty} J\Pi_K(x^i)\;|\; x^i\rightarrow x,\; J\Pi_K(x^i) \text{ exists}  \},
\end{eqnarray*}
where $J\Pi_K(x)$ denotes the Jacobian  of  $\Pi_K(\cdot)$  at $x\in \R^d$ if it exists. Then, for any $x\in \R^d$,  the Clarke generalized Jacobian of $\Pi_K(x)$, namely, $ \partial \Pi_K(x) $, is defined as the convex hull of $ \partial_B \Pi_K(x) $. The following proposition gives the concrete expression of the elements in $ \partial_B \Pi_K(x) $. We refer the reader to~\cite{PangSS2003, Kanzowsemismoothsocp2006,OutrataCoderivative2008} for more details.

\begin{proposition}
	\label{prop-projsoc}
	Given an arbitrary point $x=(x_0,x_t)^\top\in \R^d$, each element $V\in  \partial_B \Pi_K(x) $ has the following representations:
	\begin{enumerate}
		\item If $x_0\neq \pm\norm{x_t}$, $\Pi_K(\cdot)$ is continuously differentiable near $x$ with
		\begin{eqnarray*}
			J\Pi_K(x) =
			\begin{cases}
				0   & x_0< -\norm{x_t}, \\[5pt]
				I_d &  x_0>\norm{x_t}, \\
				\frac{1}{2}
				\begin{pmatrix}
					1 & \frac{x_t^\top}{\norm{x_t}} \\
					\frac{x_t}{\norm{x_t}} & (1+ \frac{x_0}{\norm{x_t}})I_{d-1} - \frac{x_0}{\norm{x_t}^3}x_tx_t^\top
				\end{pmatrix} & -\norm{x_t}< x_0< \norm{x_t}.
			\end{cases}
		\end{eqnarray*}
		\item If $x_t\neq 0$ and $x_0 = \norm{x_t}$, then
		\begin{eqnarray*}
			V\in
			\left\{
			I_d, \frac{1}{2}
			\begin{pmatrix}
				1 & \frac{x_t^\top}{\norm{x_t}} \\
				\frac{x_t}{\norm{x_t}} & 2I_{d-1} - \frac{x_t}{\norm{x_t}}\frac{x_t^\top}{\norm{x_t}}
			\end{pmatrix}
			\right\}.
		\end{eqnarray*}
		\item If $x_t\neq 0$ and $x_0 = -\norm{x_t}$, then
		\begin{eqnarray*}
			V\in
			\left\{
			{\bf 0}, \frac{1}{2}
			\begin{pmatrix}
				1 & \frac{x_t^\top}{\norm{x_t}}
				\\ \frac{x_t}{\norm{x_t}} & \frac{x_t}{\norm{x_t}}\frac{x_t^\top}{\norm{x_t}}
			\end{pmatrix}
			\right\}.
		\end{eqnarray*}
		\item If $x_t = 0$ and $x_0 = 0$, then
		\begin{eqnarray*}
			V\in \Big\{ {\bf 0},  I_d  \Big\}\, \bigcup\,
			\left\{
			\frac{1}{2}
			\begin{pmatrix}
				1 & \omega^\top
				\\ \omega & (1+\rho)I_{d-1} - \rho \omega\omega^\top
			\end{pmatrix}\;:\; |\rho|\leq 1,\; \norm{\omega} = 1 \right\}.
		\end{eqnarray*}
	\end{enumerate}
\end{proposition}
Recall that $\cK = \cK^{n_1}\times \cdots \times \cK^{n_r} \in \R^n$ is the Cartesian product of $r$ second-order cones. It is clear that for any $x=(x_1;\cdots;x_r)\in \R^n$,
\begin{eqnarray*}
	V:={\rm Diag}(V_1,\cdots,V_r)\in \partial_B\Pi_{\cK}(x),\quad V_j\in \partial_B\Pi_{\cK^{n_j}}(x_j),\quad 1\leq j\leq r.
\end{eqnarray*}

To apply the semismooth Newton method for solving the ALM subproblems presented later in the paper, we also need the concept of semismoothness.
\begin{definition}
	\label{def-semismooth}
	Let $\Phi:\W\rightarrow \Y$ be a locally Lipschitz continuous function on the open set $\mathcal{O}\subseteq \W$.
	$\Phi$ is said to be semismooth at a point $x\in \mathcal{O}$ if $\Phi$ is directionally differentiable at $x$ and for any $V\in \partial \Phi(x+\Delta x)$,
	\begin{eqnarray*}
		\Phi(x+\Delta x) - \Phi(x) - V\Delta x = o(\norm{\Delta x}),\quad \Delta x\rightarrow 0.
	\end{eqnarray*}
	$ \Phi $ is said to be strongly semismooth at $x\in \mathcal{O}$ if $ \Phi $ is semismooth at $x$ and for any $V\in \partial \Phi(x+\Delta x)$,
	\begin{eqnarray*}
		\Phi(x+\Delta x) - \Phi(x) - V\Delta x = o(\norm{\Delta x}^2),\quad \Delta x\rightarrow 0.
	\end{eqnarray*}
	$ \Phi $ is said to be a (strongly) semismooth function on $\mathcal{O}$ if it is (strongly) semismooth for every point $x\in \mathcal{O}$.
\end{definition}

The next lemma shows that $\Pi_{\cK}(\cdot)$ is strongly semismooth on $ \mathbb{R}^{n} $. For a proof of this lemma, see~\cite{Chensmoothingnewton2003,Hayashisocpcp2005}.
\begin{lemma}
	\label{lem-semismooth}
	The projection mapping $\Pi_\cK(\cdot)$ is strongly semismooth everywhere.
\end{lemma}

\section{Convergence results of the ALM}
\label{sec-ALM}
In this section, we analyze the convergence properties of the ALM applied to problem $ {\rm (P)} $. Even though the theory has been highly developed, we present certain important results here to make our paper self-contained.

Let $\sigma>0$ be a given penalty parameter. The augmented Lagrangian function associated with problem $ {\rm (P)} $ for any $ (x,y)\in \X\times \R^n $ is defined as
\begin{equation*}
L_\sigma(x,y):=f^0(x)+\frac{1}{2\sigma}
\left(
\norm{\sigma \big( -Hx_1+A^\top x_2+x_3-c\big)+y}^2 - \norm{y}^2
\right) .
\end{equation*}
At the $ (k+1) $-th iteration, for a given sequence of penalty parameters $0<\sigma_k\uparrow \sigma_\infty\leq \infty$ and an initial point $y^0\in \R^n$, the inexact ALM performs the following scheme:
\begin{eqnarray}
\label{alg-alm}
\left\{
\begin{aligned}
& x^{k+1}:=(x_1^{k+1},x_2^{k+1},x_3^{k+1})\;\approx\; {\rm arg}\min_x \;\{ f_k(x):=L_{\sigma_k}(x,y^k) \},\\
& y^{k+1} \;:=\; y^k+\sigma(-Hx_1^{k+1}+A^\top x_2^{k+1}+x_3^{k+1}-c),\quad k\geq 0.
\end{aligned}
\right.
\end{eqnarray}

The rate of convergence for the ALM can be obtained by considering its connection with the dual proximal point algorithm (PPA). This connection was explored in Rockafellar's classical papers~\cite{rockafellar1976augmented,rockafellar1976ppa}. More specifically, by combining Theorem 4 and Theorem 5 in~\cite{rockafellar1976augmented}, one obtains the following fundamental convergence result for ALM.

\begin{theorem}
	\label{thm-convergeppaalm}
	Assume that ${\rm SOL_D}$ is nonempty, i.e., $T_g^{-1}(0) \neq \emptyset$. Let $\{(x^k,y^k)\}$ be the infinite sequence generated by the ALM in~\eqref{alg-alm} under the criterion for inexact computation,
	\[
	{\rm (A)}\quad f_k(x^{k+1}) - \inf \;f_k\leq \frac{\epsilon_k^2}{2\sigma_k},
	\]
	where $\{\epsilon_k\}$ is a summable and nonnegative sequence in $\R$. Then the whole sequence $\{y^k\}$ converges to some $y^\infty\in {\rm SOL}_D$.
	
	If $T_g^{-1}$ is Lipschitz continuous at the origin with modulus $\kappa_g >0$ and the ALM is also executed under the criterion
	\[
	{\rm (B)} \quad f_k(x^{k+1}) - \inf\; f_k \leq \frac{\delta_k^2}{2\sigma_k}\norm{y^{k+1} - y^k}^2
	\]
	with a summable and nonnegative sequence $\{\delta_k\}$. Then $y^k\rightarrow y^\infty$ as $ k \rightarrow\infty $, where in this case $y^\infty$ is the unique solution for problem (D). Furthermore, it holds that
	\[
	\norm{y^{k+1} - y^\infty}\leq  \frac{\kappa_g (\kappa_g^2 + \sigma_k^2)^{-1/2} + \delta_k}{1-\delta_k} \norm{y^k - y^\infty}
	\]
	for all $k$ sufficiently large.
\end{theorem}

\begin{remark}
	\label{rmk-convergeppa}
	Note that the Lipschitz continuity assumption on $T_g^{-1}$ is rather restrictive, since it requires the solution set $T_g^{-1}(0)$ to be a singleton. In~\cite{luqueasymptotic1984}, Luque extended Rockafellar's original results by relaxing the Lipschitz continuity condition to the upper Lipschitz continuity condition. The latter condition is satisfied if the corresponding set-valued mapping is piecewise polyhedral (see Sun's PhD thesis~\cite{sunjiephd1986} for more discussions on these mappings).	However, in the present paper, we consider the mapping involving the non-polyhedral second-order cone; thus, more relaxed conditions might be needed.
\end{remark}

The classical convergence results for the ALM (or equivalently PPA) are of great value both theoretically and numerically. However, there are two practical issues to be resolved. First, we can only obtain the rate of convergence for the dual sequence $\{y^k\}$ generated by the ALM, but the rate of convergence for the primal sequence $ \{x^k\} $ is not known. Even though~\cite[Proposition 3]{CST2017} has provided a convergence result for $ \{x^k\} $ under the upper Lipschitz continuity condition of $T_l^{-1}$, the Lipschitz-like condition is quite restrictive as explained in \cite{CST2017}. Thus, instead of requiring the convergence of $\{x^k\}$ when designing a solver, in our opinion, a more reasonable requirement is the convergence of the KKT residual of the computed primal-dual sequence $\{(x^k,y^k)\}$. Second, the stopping criteria used in the theoretical analysis are not implementable since they require some unknown information (e.g., $\inf \,f_k$). Fortunately,  these issues are resolved in \cite{CST2017} by conducting finer analysis of the ALM applied to the dual problem. We shall summarize these results in the rest of this section.

To proceed, we first need the following definition of quadratic growth condition and assumption of Robinson constraint qualification.
\begin{definition}
	\label{def-quadgrowth}
	The quadratic growth condition holds at an optimal solution $\bar{y}\in {\rm SOL_D}$ if there exist positive constants $\kappa$ and $\epsilon$ such that
	\begin{eqnarray}
	-g^0(y)\geq -g^0(\bar{y}) + \kappa \,{\rm dist}^2(y,{\rm SOL_D})\quad \forall y\in \mathbb{B}_\epsilon(\bar{y})\cap \left\{y\in \R^n\;|\; Ay = b \right\}.
	\end{eqnarray}
\end{definition}

\begin{assumption}
	\label{assumpt-RCQ}
	The solution set $ {\rm SOL_{D}} $ for the problem $ {\rm (D)} $ is non-empty and the following Robinson constraint qualification (RCQ) of  the problem $ {\rm (D)} $ hold at some $\bar{y}\in  {\rm SOL_{D}} :$
	\begin{eqnarray*}
		0\in {\rm int}
		\left\{
		\begin{pmatrix}
			A\bar{y} - b\\ \bar{y}
		\end{pmatrix}
		+
		\begin{pmatrix}
			A \\ I_n
		\end{pmatrix}
		\R^n
		-
		\begin{pmatrix}
			\{0\}
			\\ \cK
		\end{pmatrix}
		\right\}.
	\end{eqnarray*}
\end{assumption}
By~\cite[Theorem 3.9]{bonnans2013perturbation}, the optimal solution set $ {\rm SOL_P} $ to the problem $ {\rm (P)} $ is nonempty and bounded under Assumption~\ref{assumpt-RCQ}.

For any $k\geq 0$, $y^k \in\R^n$, $x_1\in \R^n$ and $x_2\in \R^m$, denote
\begin{eqnarray}
\label{eq-ywse}
\left\{
\begin{aligned}
&\tilde{y}^k(x_1,x_2):=
\Pi_{\cK}\left[y^k+\sigma_k\big(-Hx_1+A^\top x_2-c\big)\right],\\
&\tilde{x}^k(x_1,x_2):=
\big(x_1,x_2, \tilde{y}^k(x_1,x_2)\big)^\top\in \X, \\
&e^k(x_1,x_2)\,:=
\begin{pmatrix}
Hx_1-H \tilde{y}^k(x_1, x_2)\\
-b+A \tilde{y}^k(x_1, x_2) \\
0
\end{pmatrix}.
\end{aligned}
\right.
\end{eqnarray}
Let $\{\hat{\epsilon}_k\}$ and $\{\hat{\delta}_k\}$ be two summable and nonnegative sequences. For inexact computations, we adopt the following stopping criteria:
\begin{eqnarray*}
	\begin{aligned}
		& {\rm (A')}\quad
		\norm{e^{k+1}} \leq
		\frac{\hat{\epsilon}_k^2/\sigma_k}{C_k}
		\min
		\left\{
		1,\frac{1}{\norm{Hy^{k+1}}+\norm{y^{k+1}-y^k}/\sigma_k+1/\sigma_k}
		\right\},    \\
		& {\rm (B')}\quad
		\norm{e^{k+1}} \leq
		\frac{(\hat{\delta}_k^2/\sigma_k)\norm{y^{k+1}-y^k}^2}{C_k}
		\min
		\left\{
		1,\frac{1}{\norm{Hy^{k+1}}+\norm{y^{k+1}-y^k}/\sigma_k+1/\sigma_k}
		\right\},
	\end{aligned}
\end{eqnarray*}
where $e^{k+1}:=e^k(x_1^{k+1},x_2^{k+1})$, $ y^{k+1}:=\tilde{y}^k(x_1^{k+1},x_2^{k+1})$, and $C_k\geq 1$ is defined as
\begin{equation*}
	C_k:=1+\norm{\tilde{x}^k(x_1^{k+1},x_2^{k+1})}+\norm{y^{k+1}}.
\end{equation*}
We can see that the above stopping criteria are truly implementable, and hence they are more useful for practical purposes than the classical ones (i.e., criteria ${\rm (A)}$ and ${\rm (B)}$).

Based on the KKT optimality condition~\eqref{kkt-condition}, we define the natural map
\begin{eqnarray}
\label{naturalmap}
R^{\rm nat}(x,y):=
\begin{pmatrix}
Hx_1-Hy\\
-b+Ay\\
x_3-\Pi_{\cK}(x_3-y)\\
Hx_1-A^\top x_2-x_3+c
\end{pmatrix}\quad
\forall\,
x = (x_1,x_2,x_3)\in \X,\;
y\in \R^n.
\end{eqnarray}
The following theorem is taken from \cite[Theorem 2]{CST2017}, which provides the R-superlinear convergence of the KKT residual.

\begin{theorem}
	\label{thm-rsuperlinear}
	Suppose that Assumption~\ref{assumpt-RCQ} holds.
	Let $\{(x^k,y^k)\}$ be an infinite sequence generated by the ALM in~\eqref{alg-alm} under the criterion ${\rm (A')}$. Then the sequence $\{y^k\}$ is bounded and converges to some $y^\infty\in{\rm SOL_D}$. Moreover, the sequence $\{x^k\}$ is also bounded with all of its limit points in ${\rm SOL_P}$.
	
	If criterion ${\rm (B')}$ is also executed in the ALM and the quadratic growth condition holds at $y^\infty$ with modulus $\hat{\kappa}_g>0$, then there exist a positive constant $\alpha$ and an integer $\bar{k}\geq 0$ such that for all $k\geq \bar{k}$, $\alpha \hat{\delta}_k<1$, and
	\begin{eqnarray*}
		{\rm dist}\,(y^{k+1},{\rm SOL_D})\leq
		\theta_k {\rm dist}\,(y^k,{\rm SOL_D}),\quad
		\norm{R^{\rm nat}(x^{k+1},y^{k+1})}\leq
		\theta_k'{\rm dist}\,(y^k,{\rm SOL_D}),
	\end{eqnarray*}
	where
	\begin{eqnarray*}
		\theta_k&:=&\frac{1}{1-\alpha \hat{\delta}_k}
		\left(
		\alpha \hat{\delta}_k + \frac{\alpha \hat{\delta}_k}{\sqrt{1+\sigma_k^2\hat{\kappa}_g^2}}
		\right),\\[5pt]
		\theta_k'&:=&\frac{1}{1-\alpha \hat{\delta}_k}
		\left(
		\max
		\left\{
		1, \frac{1}{\sigma_k}
		\right\} +
		\frac{\hat{\delta}_k^2}{\sigma_k}\norm{y^{k+1}-y^k}
		\right).
	\end{eqnarray*}
\end{theorem}

\bigskip
One can observe from the above theorem that when $\sigma_k\uparrow\sigma_\infty\leq \infty $,
\begin{eqnarray*}
	\theta_k\rightarrow
	\theta_\infty:=\frac{1}{\sqrt{1+\sigma_\infty^2\hat{\kappa}_g^2}},\quad
	\theta_k'\rightarrow
	\theta_\infty':=\max\left\{ 1,\frac{1}{\sigma_\infty} \right\}.
\end{eqnarray*}
Thus $\theta_\infty$ can be arbitrarily close to zero if $\sigma_\infty$ is sufficiently large. This implies that the linear convergence rate for the sequence $ \{\mathrm{dist}(y^{k}, \mathrm{SOL}_D)\} $ can be arbitrarily small. Moreover, since $\theta_\infty'\leq 1$,  the KKT residual also converges as rapidly as $ \{\mathrm{dist}(y^{k}, \mathrm{SOL}_D)\} $. These convergence properties may explain partially the highly efficiency of the ALM, as we shall see in our numerical experiments.

\section{Quadratic growth condition}
\label{sec-quad-growth}
In this section, we analyze the quadratic growth condition for the dual problem $ {\rm (D)} $, which serves as a sufficient condition for the  KKT residual generated by the ALM  to achieve the R-superlinear convergence rate (see Theorem~\ref{thm-rsuperlinear}). In the recent work of Cui, Ding, and Zhao \cite{CDZhao2016}, two types of sufficient conditions were proposed to ensure the quadratic growth condition. Here in this section, we will follow one of the available frameworks in \cite{CDZhao2016} to provide a sufficient condition for the quadratic growth condition under the bounded linear regularity and metric subregularity conditions.

Recall that since $H\succeq \bm{0}$, there exists a matrix $R$ such that $ H = R^\top R $. Denote $F_{D}:=\big\{ y\in \R^n \;\left\vert\;  Ay = b \right. \big\}$. Then problem $ {\rm (D)} $ can be reformulated as
\begin{equation*}
\max_{y} \;
\left\{ \left.
g^0(y):=-\frac{1}{2}\norm{Ry}^2-\inprod{c}{y} - p(y) \;\right\vert\; y\in F_{D}
\right\},
\end{equation*}
where $p(\cdot)=\delta_{\cK}(\cdot)$.
Moreover, the KKT optimality condition~\eqref{kkt-condition} can be rewritten as
\begin{eqnarray}
\label{kkt-condition2}
0\in R^\top Ry + c+\partial p(y)-A^\top x_2,\quad
Ay-b = 0\quad
\forall \; (x_2,y) \in \R^m\times \R^n.
\end{eqnarray}
Take any $\bar{y}\in {\rm SOL_{D}}$.  Denote
\begin{equation*}
\bar{\zeta}:=R \bar{y},\quad
\overline{\mathcal{V}}:=\big\{ y\in \R^n \;|\; R y= \bar{\zeta} \big\}
\end{equation*}
and define the set-valued mapping $\mathcal{G}:\X\rightrightarrows\R^n$ as
\begin{equation*}
\mathcal{G}(x):=
(\partial p)^{-1}\left(A^\top x_2 - R^\top \bar{\zeta}-c\right)\quad
\forall \; x=(x_1,x_2,x_3) \in \X.
\end{equation*}
Then, we have the following characterization for the optimal solution set $ {\rm SOL_{D}} $.
\begin{proposition}
	\label{prop-SOLD2}
	Assume that $\bar{y}\in {\rm SOL_{D}}$ and $\bar{x}= (\bar{x}_1,\bar{x}_2,\bar{x}_3)\in {\rm SOL_P}$.
	Then the optimal solution set $ {\rm SOL_{D}} $ can be characterized as
	\begin{eqnarray*}
		{\rm SOL_{D}} =
		\overline{\mathcal{V}}\cap \mathcal{G}(\bar{x}) \cap F_D.
	\end{eqnarray*}
\end{proposition}
\begin{proof}
	We only have to show that for any $y, y'\in {\rm SOL_D}$, it holds that $Ry = Ry'$. This is equivalent to saying that the value $Ry$ is invariant over $y\in {\rm SOL_D}$. However, such a fact is already well known in the literature; see, for instance~\cite{Mangasarian1988}.
\end{proof}

Next, we recall the concept of bounded linear regularity of a collection of closed convex sets. This concept is useful for analyzing error bound properties for constrained optimization problems.
\begin{definition}
	\label{def-bddlinreg}
	Let $D_1,\cdots,D_q$ be some closed convex sets in a finite dimensional Euclidean space $ \W $.
	Suppose that $D:=D_1\cap \cdots \cap D_q$ is non-empty.
	The collection $\{ D_1,\cdots,D_q \}$ is said to be boundedly linearly regular if for every bounded set $B\subseteq \W$, there exists a positive constant $\kappa$ such that
	\begin{eqnarray*}
		{\rm dist}\,(x,D) \; \leq \;
		\kappa \max
		\left\{ {\rm dist}\,(x,D_1),\cdots,{\rm dist}\,(x,D_q) \right\} \quad
		\forall \, x\in B.
	\end{eqnarray*}
\end{definition}
However, checking the condition in Definition \ref{def-bddlinreg} is not a trivial task. In~\cite[Corollary 3]{bauschke1999strong}, the authors established the following simpler sufficient condition.
\begin{proposition}
	\label{prop-bddlinreg}
	Let $D_1,\cdots,D_q$ be some closed convex sets in a finite dimensional Euclidean space $ \W $. Suppose that $D_1,\cdots,D_{q_1}$ are polyhedral for some $0\leq q_1\leq q$. Then a sufficient condition for the collection $\{ D_1,\cdots,D_q \}$ to be boundedly linearly regular is
	\begin{eqnarray*}
		\bigcap_{1\leq i\leq q_1}\,D_i
		\;\;\cap\;\;
		\bigcap_{q_1+1\leq i\leq q}\,{\rm ri}\,(D_i)\;
		\neq \;\emptyset.
	\end{eqnarray*}
\end{proposition}

We next introduce the definition of metric subregularity.
\begin{definition}
	\label{def-metricsubregular}
	A multifunction $\Phi:\W \rightrightarrows \Y $ is said to be metrically subregular at $\bar{x}\in \W$ for $\bar{v}\in \Y$ if $(x,v)\in {\rm gph}\,(\Phi)$ and there exist positive constants $\kappa$ and $\epsilon$ such that
	\begin{eqnarray*}
		{\rm dist}\,(x,\Phi^{-1}(\bar{v}))\leq
		\kappa\, {\rm dist}\,(\bar{v},\Phi(x))\quad
		\forall\, x\in \mathbb{B}_{\epsilon}(\bar{x}).
	\end{eqnarray*}
\end{definition}
For a general multifunction, it could be difficult to check the metric subregularity directly since the graph of the multifunction at the reference point may contain infinitely many points. Fortunately, when the multifunction is the subdifferential of a proper closed convex function, it has a more convenient characterization as shown in the next proposition.

\begin{proposition}
	\label{prop-metricsubregular}
	Let $\W$ be a real Hilbert space endowed with the inner product $\inprod{\cdot}{\cdot}$ and $p:\W\rightarrow (-\infty,+\infty]$ be a proper closed convex function. Let $(\bar{v},\bar{x})\in \W\times \W$ such that $\bar{v}\in \partial p(\bar{x})$. Then $\partial p$ is metrically subregular at $\bar{x}$ for $\bar{v}$ if and only if there exist positive constants $\kappa $ and $\epsilon$ such that
	\begin{eqnarray*}
		p(x) \geq p(\bar{x}) + \inprod{\bar{v}}{x-\bar{x}} + \kappa\, {\rm dist}^2\,(x,(\partial p)^{-1}(\bar{v}))\quad
		\forall\, x\in \mathbb{B}_\epsilon(\bar{x}).
	\end{eqnarray*}
\end{proposition}
The proof of Proposition~\ref{prop-metricsubregular} can be found in~\cite[Theorem 3.3]{aragon2008characterization}. Next proposition states that $\partial p(\cdot)=\partial \delta_\cK(\cdot) = \mathcal{N}_{\mathcal{K}}(\cdot)$ is indeed metrically subregular.

\begin{proposition}
	\label{prop-indicator-K}
	Let $\cK = \cK^{n_1}\times\cdots\times \cK^{n_r}\subseteq \R^n$ be the Cartesian product of some second-order cones with $\cK^{n_i}\subseteq \R^{n_i},\;i = 1,\cdots,r$, and $ n = n_1+\cdots+n_r$. For any $(x,v)\in {\rm gph}(\partial \delta_\cK)$ i.e., {$v\in \cN_{\cK}(x)$}, $\partial \delta_\cK(\cdot)$ is metrically subregular at $x$ for $v$.
\end{proposition}
\begin{proof}
	Since the metric subregularity of $\mathcal{N}_{\mathcal{K}}$ is implied by the metric subregularity of each $\mathcal{N}_{\mathcal{K}^{n_i}}$ for $i = 1,\dots, r$, we only need to check that for a standard second-order cone $K$, $\mathcal{N}_K(\cdot)$ is metrically subregular at any point on its graph. The latter has been shown in \cite{sunying2019} as a special case of the results for the $p$-order conic constraint system. Thus, the proof is completed.
\end{proof}

After all the previous preparations, we are now able to provide a sufficient condition for the quadratic growth condition for problem $ {\rm (D)} $ to hold. The next theorem, which is taken from~\cite{CDZhao2016}, provides a general framework to establish the sufficient condition for the quadratic growth condition. To make the paper self-contained and to explain the idea more clearly, we provide a proof that is restricted to SOCPs.

\begin{theorem}
	\label{thm-sufficient}
	Assume that $ {\rm SOL_{D}} $ is nonempty and that there exists $\bar{x}\in {\rm SOL_P}$ such that the collection $\big\{\overline{\mathcal{V}},\mathcal{G}(\bar{x})\big\}$ is boundedly linearly regular. Then the quadratic growth condition holds for problem $ {\rm (D)}$ at any point $\bar{y}\in {\rm SOL_{D}}$.
\end{theorem}

\begin{proof}
	Let $\bar{y} \in {\rm SOL_{D}}$, $\bar{x} = (\bar{x}_1,\bar{x}_2,\bar{x}_3)\in {\rm SOL_P}$ and $\epsilon>0$. Then for any $y\in F_{D}\cap \mathbb{B}_{\epsilon}\big(\bar{y}\big)$ we have that there exist $\kappa_2>0$ and $\kappa_3>0$ such that
	\begin{eqnarray*}
		\begin{aligned}
			{\rm dist}^2 \big( y,{\rm SOL_{D}} \big)
			& =  {\rm dist}^2\,\big( y,\overline{\mathcal{V}}	
			\cap\mathcal {G}(\bar{x})  \big)\\
			&\leq \kappa_2\,
			\big[ {\rm dist}^2\,\big( y,\overline{\mathcal{V}} \big) +{\rm dist}^2\,\big( y, \mathcal{G}(\bar{x}) \big)  \big] \\
			&\leq \kappa_3\, \big[ \norm{R y-\bar{\zeta}}^2
			+ {\rm dist}^2\,\big( y, (\partial p)^{-1}(A^\top x_2-R^\top \bar{\zeta} - c) \big) \big],
		\end{aligned}
	\end{eqnarray*}
	under the assumption that $\big\{\overline{\mathcal{V}},\mathcal{G}(\bar{x})  \big\}$ is boundedly linearly regular. Note that in the last inequality of above, the first term makes use of Hoffman's error bound~\cite{Hoffman1952}.
	
	By Proposition~\ref{prop-metricsubregular}, Proposition~\ref{prop-indicator-K} and $ (\bar{y},A^\top \bar{x}_2- R^\top \bar{\zeta}-c)\in {\rm gph}(\partial p) $, we know that there exists $\kappa_p>0$ such that for any $y\in \mathbb{B}_{\epsilon}(\bar{y})$, it holds that (by shrinking $\epsilon$ if necessary),
	\begin{eqnarray*}
		\begin{aligned}
			p(y)-p(\bar{y}) &  \geq
			\inprod{A^\top  \bar{x}_2- R^\top \bar{\zeta}-c}{y-\bar{y}}+\kappa_p{\rm dist}^2\,(y,(\partial p)^{-1}(A^\top  \bar{x}_2- R^\top \bar{\zeta}-c)).
		\end{aligned}
	\end{eqnarray*}
	
	By combining all the obtained inequalities, we have that for any $y\in F_{D}\cap \mathbb{B}_{\epsilon}(\bar{y})$,
	\begin{eqnarray*}
	\begin{aligned}
	-g^0(y)
	&\;=\; \frac{1}{2}\norm{Ry}^2+\inprod{c}{y}+p(y)\\
	&\;\geq\; \frac{1}{2}\norm{\bar{\zeta}}^2
	+ \inprod{\bar{\zeta}}{Ry-\bar{\zeta}}
	+ \frac{1}{2}\,\norm{Ry-\bar{\zeta}}^2 +\inprod{c}{y} \\
	&\;\quad\; + p(\bar{y})
	+ \inprod{A^\top  \bar{x}_2-R^\top\bar{\zeta} - c}{y-\bar{y}}
	+\kappa_p\,{\rm dist}^2\,\big(y,(\partial p)^{-1}
	(A^\top\bar{x}_2- R^\top \bar{\zeta}-c)\big) \\
	&\;=\; -g^0(\bar{y}) + \frac{1}{2}\,\norm{Ry-\bar{\zeta}}^2
	+ \kappa_p\,{\rm dist}^2\big(y,(\partial {p})^{-1}(A^\top \bar{x}_2
	- R^\top \bar{\zeta}-c)\big)\\
	&\;\geq\; -g^0(\bar{y})
	+ \kappa_3^{-1}\min\{\kappa_p,\frac{1}{2}\}\, {\rm dist}^2\,(x,{\rm SOL_{D}}),
	\end{aligned}
	\end{eqnarray*}
	which is exactly the quadratic growth condition for problem $ {\rm (D)}$. Therefore, the proof is completed.
\end{proof}

By the definitions of $\overline{\mathcal{V}}$ and ${\rm F_D}$, it is obvious that both sets are polyhedral. However, $\mathcal{G}(\cdot)$ is not always polyhedral. Indeed, let $\bar{x}_3:= -A^\top\bar{x}_2+ R^\top \bar{\zeta}+c = \left((\bar{x}_3)_1,\cdots,(\bar{x}_3)_r\right)^\top \in\R^n$, since (see, e.g.,~\cite{JFBonnans2005})
\begin{eqnarray*}
	(\partial \delta_{\cK^{n_i}})^{-1}\,(-(\bar{x}_3)_i) = \cN_{(\cK^{n_i})^\circ}(-(\bar{x}_3)_i) =
	\begin{cases}
		\{0\} & (\bar{x}_3)_i\in {\rm int}\,\cK^{n_i},\\
		\cK^{n_i} & (\bar{x}_3)_i = 0,\\
		\R_+((\bar{x}_3)_{i,0},-(\bar{x}_3)_{i, t}) &
		(\bar{x}_3)_i\in {\rm bd}\,\cK^{n_i} \backslash \{0\},
	\end{cases}
\end{eqnarray*}
we can see that when $n_i\geq 3$, $\mathcal{G}(\bar{x})$ is polyhedral if and only if $(\bar{x}_3)_i\neq 0,\,\forall \, 1\leq i\leq r$. As a consequence, given $\bar{x}_2\in \R^m$, let $J$ be the index set defined as $J:=\left\{i\;|\;(\bar{x}_3)_i = 0\right\}$, if there exists $\bar{y}=(\bar{y}_1,\cdots,\bar{y}_r)^\top\in \R^n$ such that $(\bar{x}_2,\bar{y})$ solves the KKT system~\eqref{kkt-condition2} and $\bar{y}_i\in {\rm int}\,\cK^{n_i}$, $\forall\,i\in J$. Then by Proposition~\ref{prop-bddlinreg}, the collection $\{\overline{\mathcal{V}}, \mathcal{G}(\bar{x}) \}$ is boundedly linearly regular, and hence the quadratic growth condition for the problem $ {\rm (D)}$ holds at any optimal solution. The aforementioned conclusion on $(\bar{x}_2,\bar{y})$ is summarized as follows.

\begin{corollary}
	\label{cor-sc}
	Let $(\bar{x}_2,\bar{y})$ be a solution of the KKT system~\eqref{kkt-condition2} and  $\bar{x}_3 = -A^\top\bar{x}_2+R^\top \bar{\zeta}+c$. If for each block with $n_i\ge 3$, $i = 1,\cdots, r$, $((\bar{x}_3)_i,\bar{y}_i)$ satisfies the  strictly complementary condition: $\bar{y}_i+(\bar{x}_3)_i\in {\rm int}\, \cK^{n_i}$. Then the quadratic growth condition holds at any solution of the problem $ {\rm (D)}$.
\end{corollary}
\begin{proof}
	By~\cite[Corollary 24]{alizadehsocp2003}, we know that for each block with $n_i\ge 3$, $1\leq i\leq r$, $((\bar{x}_3)_i,\bar{y}_i)$ satisfies the strictly complementary condition either when both $\bar{y}_i$ and $(\bar{x}_3)_i$ are nonzero and in the ${\rm bd}\, \cK^{n_i}$, or when one of them is zero and the other is in the interior of $\cK^{n_i}$. Then by the above discussions, the conclusion can be derived in a straight-forward manner.
\end{proof}

\section{Solving the ALM subproblem by an inexact semismooth Newton method}
\label{sec-issn}
In this section, we propose an inexact semismooth Newton method for solving subproblems arsing from the inexact ALM in~\eqref{alg-alm} applied to the problem $ {\rm (P)} $.

For given $y$ and $\sigma$, denote $ \tilde{y}(x_1,x_2,y):= Hx_1-A^\top x_2 - y/\sigma + c$. Recall that the exact ALM subproblem is given by
\begin{equation}
\label{alm-subprob}
(x_1^+,x_2^+,x_3^+) =
\underset{(x_1,x_2,x_3)\in \X}{{\rm argmin}}
\left\{ \begin{array}{l}
\frac{1}{2}\inprod{x_1}{Hx_1} -
\inprod{b}{x_2}+\delta_\cK(x_3)
\\
+
\frac{1}{2\sigma}
\left(
\norm{\sigma(x_3 -  \tilde{y}(x_1,x_2,y))}^2-\norm{y}^2
\right)
\end{array}
\right\}.
\end{equation}
By simple calculations, we have
\begin{eqnarray}
\label{eq-x3}
x_3^+ = \Pi_{\cK}\left( \tilde{y}(x_1^+,x_2^+,y) \right).
\end{eqnarray}
Therefore, by using the Moreau identity, we obtain that to solve the problem~\eqref{alm-subprob}, it is equivalent to solve
\begin{equation}
\label{alm-subprob2}
\min_{x_1,x_2}\;\psi(x_1,x_2) :=
\frac{1}{2}\inprod{x_1}{Hx_1} - \inprod{b}{x_2} +
\frac{1}{2\sigma}
\left(
\norm{
	\Pi_{\cK}
	\left[
	-\sigma \tilde{y}(x_1,x_2,y)
	\right]
}^2 -
\norm{y}^2 \right).
\end{equation}
Once $x_1^+$ and $x_2^+$ have been computed, we can obtain $x_3^+$ via \eqref{eq-x3}. Furthermore, to solve the above unconstrained minimization problem with respect to $(x_1,x_2)\in {\rm Ran}(H)\times \R^m$, it is equivalent to solve the following system of nonsmooth equations:
\begin{equation*}
\grad \psi(x_1,x_2) =
\begin{pmatrix}
Hx_1 - H\Pi_{\cK}
\left[ -\sigma \tilde{y}(x_1,x_2,y) \right] \\
-b + A \Pi_{\cK}\left[ -\sigma \tilde{y}(x_1,x_2,y) \right]
\end{pmatrix} = 0,\quad (x_1,x_2)\in \mathrm{Ran}(H)\times \mathbb{R}^m.
\end{equation*}
Since $\Pi_{\cK}(\cdot)$ is strongly semismooth everywhere (by Proposition \ref{lem-semismooth}), it is desirable to apply a semismooth Newton method to solve the above system of nonsmooth equations as one could expect a superlinear \textit{or} even quadratic convergence rate. To this end, for any $(x_1,x_2)\in {\rm Ran}(H)\times \R^m$, we define
\begin{eqnarray*}
	\hat{\partial}^2\psi(x_1,x_2):=
	\begin{pmatrix}
		H & \\
		& 0
	\end{pmatrix} + \sigma
	\begin{pmatrix}
		H \\
		-A
	\end{pmatrix}
	\partial \Pi_{\cK}\left[ -\sigma \tilde{y}(x_1,x_2,y) \right]
	\begin{pmatrix}
		H &-A^\top
	\end{pmatrix}.
\end{eqnarray*}
Then $\hat{\partial}^2\psi(x_1,x_2)$ can serve as a replacement of the (hard-to-characterize) generalized Hessian of $ \psi $ at $ (x_1,x_2) $, namely, $ \partial^2\psi(x_1,x_2) $, in the sense that for any $d_1\in{\rm Ran}(H)$ and $d_2\in \R^m$,
\begin{eqnarray*}
	\hat{\partial}^2\psi(x_1,x_2)
	\left(
	\begin{pmatrix}
		d_1\\
		d_2
	\end{pmatrix}
	\right) =
	\partial^2 \psi(x_1,x_2)
	\left(
	\begin{pmatrix}
		d_1\\
		d_2
	\end{pmatrix}
	\right).
\end{eqnarray*}

Next we present the well-known inexact semismooth Newton method in \cite{ZhaoSTNewtonCGALM2010} to solve \eqref{alm-subprob2} as in Algorithm~\ref{alg-iSSN}.
\begin{algorithm}[t]
	\quad \quad Given  $ \hat{\nu}\in (0,1) $, $\tau\in (0,1]$, $\tau_1,\tau_2\in (0,1)$, and
	$\mu\in (0,1/2)$, $\delta\in(0,1)$.
	Choose $(x_1^0,x_2^0)\in {\rm Ran}(H)\times \R^m$.
	Perform the following iterations for $j = 0,1,2,\cdots,$
	\begin{description}
		\item[Step 1.] Set $ \epsilon_j:=\tau_1\min\left\{\tau_2, \norm{\grad \psi(x_1^j,x_2^j)} \right\} $ and $\nu_j:=\min\left\{  \hat{\nu},\norm{\grad \psi(x_1^j,x_2^j)}^{1+\tau}\right\}$. Find $(d_1^j,d_2^j)\in  {\rm Ran}(H)\times \R^m$ by solving the following linear system approximately
		\begin{eqnarray*}
			M_j
			\begin{pmatrix}
				d_1\\
				d_2
			\end{pmatrix} + \epsilon_j
			\begin{pmatrix}
				0\\
				d_2
			\end{pmatrix} +
			\grad \psi(x_1^j,x_2^j) = 0,\quad
			M_j \in \hat{\partial}^2\psi(x_1^j,x_2^j)
		\end{eqnarray*}
		in the sense that
		\begin{eqnarray*}
			\norm{
				M_j
				\begin{pmatrix}
					d_1^j\\ d_2^j
				\end{pmatrix} + \epsilon_j
				\begin{pmatrix}
					0\\
					d_2^j
				\end{pmatrix} +
				\grad \psi(x_1^j, x_2^j)
			} \leq \nu_j.
		\end{eqnarray*}
		\item[Step 2.] Set $\alpha_j=\delta^{m_j}$ where $m_j$ is the smallest non-negative integer $m$ for which
		\begin{eqnarray*}
			\psi(x_1^j+\delta^md_1^j,x_2^j+\delta^md_2^j) \leq
			\psi(x_1^j,x_2^j) + \mu \delta^m \inprod{\grad \psi(x_1^j,x_2^j)}{\begin{pmatrix}
					d_1^j\\
					d_2^j
			\end{pmatrix}}.
		\end{eqnarray*}
		\item[Step 3.] Set $x_1^{j+1} = x_1^j+\alpha_j d_1^j$ and $x_2^{j+1} = x_2^j+\alpha_j d_2^j$.
	\end{description}
	\caption{{\bf Algorithm iSSN: An inexact semismooth Newton method (iSSN($y,\sigma $)).}}
	\label{alg-iSSN}
\end{algorithm}

The convergence of Algorithm~\ref{alg-iSSN} is given by the next theorem under the following assumption.
\begin{assumption}
	\label{assump-slater}
	The linear mapping $A:\R^n\rightarrow \R^m$ is onto, and there exists $ \hat{y}\in {\rm int}\,\cK$ such that $A\hat{y} = b$.
\end{assumption}
\begin{theorem}
	\label{thm-convergeissn}
	Suppose that Assumption~\ref{assump-slater} holds.
	Then Algorithm~\ref{alg-iSSN} generates a bounded sequence $\{(x_1^j,x_2^j)\}$
	such that any of its accumulation point is an optimal solution to problem~\eqref{alm-subprob2}.
\end{theorem}
Readers may refer to~\cite[Theorem 3.4]{ZhaoSTNewtonCGALM2010} for a proof of Theorem~\ref{thm-convergeissn}. To obtain a fast superlinear convergence rate {\it or} even a quadratic convergence rate of Algorithm~\ref{alg-iSSN}, one needs the positive definiteness of the coefficient matrix in the linear system at the solution point. Establishing conditions that ensure the positive definiteness of the coefficient matrix is important for the convergence analysis. Next theorem provides the convergence rate of the algorithm under the constraint nondegeneracy condition, whose proof can be done by combining the results from~\cite[Proposition 3.1, Theorem 3.2]{LiSTQSDPNAL2018} and~\cite[Proposition 3.2, Theorem 3.5]{ZhaoSTNewtonCGALM2010}.
\begin{theorem}
	\label{thm-rateissn}
	Suppose that Assumption~\ref{assump-slater} holds. Let $(\hat{x}_1,\hat{x}_2)$ be an accumulation point of the infinite sequence $\{(x_1^j,x_2^j)\}$ generated by Algorithm ${\rm iSSN}$ for problem~\eqref{alm-subprob2}. Let $\hat{y}:=\Pi_{\cK}\left(-\sigma \tilde{y}(\hat{x}_1, \hat{x}_2,y) \right)$. Assume that the following constraint nondegeneracy condition holds
	\begin{eqnarray*}
		A\,{\rm lin}\,\Big( \cT_{\cK}(\hat{y})  \Big) = \R^m,
	\end{eqnarray*}
	where ${\rm lin}\,\Big( \cT_{\cK}(\hat{y})  \Big)$ denotes the lineality space of the tangent cone of $\cK$ at $\hat{y}$. Then, the whole sequence $\{(x_1^j,x_2^j)\}$ converges to $(\hat{x}_1,\hat{x}_2)$ and
	\begin{eqnarray*}
		\norm{(x_1^{j+1},x_2^{j+1}) - (\hat{x}_1,\hat{x}_2)} =
		O\left(
		\norm{(x_1^j,x_2^j) - (\hat{x}_1,\hat{x}_2)}^{1+\tau}
		\right).
	\end{eqnarray*}
\end{theorem}

\begin{remark}
	The constraint nondegeneracy condition in the above theorem could be hard to verify since the accumulation point $(\hat{x}_1, \hat{x}_2)$ is usually not known. Fortunately, for some special problems one may check that this condition holds at any feasible solution. For such an example, see Theorem \ref{thm-nondegmeb} in Section \ref{subsec-meb} on solving MEB problems.
\end{remark}

Note that under the constraint nondegeneracy condition, one can  show that every element in $\hat{\partial}^2\psi(\hat{x}_1, \hat{x}_2)$ is self-adjoint and positive definite on $\mathrm{Ran}(H)\times \mathbb{R}^m$; see \cite[Theorem 3.2]{LiSTQSDPNAL2018}. It is also clear that if $H$ is not positive definite on $\mathbb{R}^n$, then $\mathrm{Ran}(H)\neq \mathbb{R}^n$. Thus, if  $\mathrm{Ran}(H)$ is replaced by any linear subspace of $\mathbb{R}^n$ strictly containing $\mathrm{Ran}(H)$ in the formulation of problem \textrm{(P)}, then the local fast convergence rate for Algorithm iSSN will be lost. As a result, the restriction  $x_1\in \mathrm{Ran}(H)$ in problem \textrm{(P)} in fact plays a crucial role in our algorithmic framework. We will discuss later in Section \ref{subsec-implementation} on how to implement the restriction $(d_1, d_2)\in \mathrm{Ran}(H)\times \mathbb{R}^m$ when solving the linear system in Algorithm iSSN.

We end this section by emphasizing that our ALM, equipped with a semismooth Newton method for solving the ALM subproblems, is an inner-outer loop algorithm. By our convergence analysis, both the inner loop and the outer loop have fast convergence rates under some technical assumptions. Thus, our present algorithm is a ``fast+fast'' algorithm.

\section{Numerical implementation and  experiments}
\label{sec-numexp}
In this section, we aim to design an efficient solver for the following  SOCP problem
\begin{eqnarray}
\label{QLP}
\min_{x_1,x_2,x_3} \;
\left\{
\frac{1}{2}\inprod{x_1}{Hx_1}-\inprod{b}{x_2}\;\left\vert\;
\begin{array}{c}
-Hx_1 + A^\top x_2+x_3 = c,\quad  \\
x_3 = \big((x_3)_0,(x_3)_1,\cdots,(x_3)_r\big) \in \cK
\end{array}
\right.
\right\},
\end{eqnarray}
where $\cK:=\R_+^{n_0} \times \cK^{n_1}\times \cdots \times \cK^{n_r}$ with $\cK^{n_i}$ $(n_i\geq 3)$ being the second-order cone in $\R^{n_i}$ for $1\leq i\leq r$, $c=(c_0,c_1,\cdots,c_r)^\top\in \R^{n_0}\times \R^{n_1}\times\cdots \times \R^{n_r}$, $b\in \R^m$, $A\in \R^{m\times n}$, and $H\in\S_+^n$  are given data
with
$ n = n_0+n_1+\cdots+n_r.$
Moreover,  if we treat $A$ as a linear mapping such that
$A:\R^n\rightarrow \R^m$, then it has the following form:
\begin{eqnarray*}
	Ay  := \sum_{i=0}^rA_iy_i,\quad
	A_i\in\R^{m\times n_i},\;
	0\leq i \leq r
	\quad \forall \; y = (y_0;y_1;\cdots;y_r)\in\R^n.
\end{eqnarray*}

Note that in~\eqref{QLP}, we consider additionally a nonnegative constraint since it often appears in real world applications. However, all the theoretical development in the previous parts of the paper can easily be extended to include the additional nonnegative constraint since the cone $\R^{n_0}_+$ is polyhedral.

In the remaining part of this section, we first discuss some implementation details for the proposed ALM. Next, we apply our SOCP solver to solve MEB problems, trust-region subproblems, square-root Lasso problems, and some linear SOCPs problems in the DIMACS challenge data set. We also mention here that the purpose of our numerical experiments is to compare the efficiency of our proposed ALM against other well-known linear SOCP solvers. Therefore, we do not compare the performance of our solver with specialized solvers for each application that are presented in the rest of this section.

\subsection{On the efficient implementation of the ALM for SOCP}
\label{subsec-implementation}
In this subsection, we present some implementation details for our ALM solver. In particular, we discuss how to solve the Newton systems efficiently when the input data possesses certain sparsity structures.

First, we consider solving systems arising in linear SOCPs. Let us focus on the case when $A$ is a sparse matrix. For any given $(y, x_2)$ and $\sigma>0$, it is shown in section~\ref{sec-issn} that the crucial task for solving the ALM subproblem is to solve a linear system in the following form:
	\begin{eqnarray}
		\label{lineareq}
		Md :=
		\left(\epsilon I_m + \sum_{i=0}^r M_i\right)d = {\rm rhs},
		\quad d\in \mathbb{R}^m,
	\end{eqnarray}
where $M_i:=A_iV_iA_i^\top $, for $1\leq i\leq r$, $\epsilon$ is a small positive  number, ${\rm rhs}$ is a given vector, and
	\begin{eqnarray*}
		V_0 & \in & \partial_B \Pi_{\R_+^n}
		\left(
			y_0+\sigma(A_0^\top x_2-c_0)
		\right),\\ [5pt]
		V_i & \in & \partial_B \Pi_{\cK^{n_i}}
		\left(
			y_i+\sigma(A_i^\top x_2-c_i)
		\right),\quad
		1\leq i \leq r.
	\end{eqnarray*}
From the description of the elements in $\partial_B\Pi_{\cK^{n_i}}(\cdot)$ presented in section~\ref{sec-pre}, we can see that if $V_i$ ($1\leq i\leq r$) takes the following form:
\begin{eqnarray*}
	V_i = \frac{1}{2}
	\begin{pmatrix}
		1 		 & \omega_i^\top  \\[5pt]
		\omega_i & (1+\rho_i)I_{n_i-1}-\rho_i \omega_i\omega_i^\top
	\end{pmatrix},\quad
	|\rho_i| \leq 1,\;
	\norm{\omega_i}= 1,
\end{eqnarray*}
then $M_i$ can be rewritten as
\begin{eqnarray*}
	M_i \;=\; A_iV_iA_i^\top \;= \;
	\frac{1+\rho_i}{2} A_i A_i^\top + \frac{1}{2}
	\left(A_{i,1},\; A_{i,2}w_i\right) \left(\begin{array}{cc} -\rho_i & 1 \\ 1 & -\rho_i\end{array} \right)
	\left(A_{i,1},\; A_{i,2}w_i\right)^\top,
\end{eqnarray*}
where $A_i = (A_{i,1},A_{i,2})$ with $A_{i,1} \in \R^m$ and $A_{i,2} \in \R^{m\times {(n_i-1)}}$.

The presence of the outer-product terms in the formulation of the matrix $M_i=A_iV_iA_i^\top $ can cause numerical issue in the following sense. If the vector $A_{i,2}w_i$ is dense, even when $A_{i}A_{i}^\top $ is a sparse matrix, $M_i$ will still be a dense matrix. In this case, directly solving~\eqref{lineareq} based on Cholesky factorization will be time consuming. To overcome the aforementioned issue, we will apply the following dense-column handling technique to exploit the possibly sparse part of the matrix $M_i$.

Let us assume that the coefficient matrix $M$ can be written as $M = M_{\rm sp}+UDU^\top$ where $M_{\rm sp}$ is a sparse symmetric positive definite matrix, $U$ has only a few columns, and $D$ is an invertible diagonal matrix. Then we can solve the linear system~\eqref{lineareq} by solving the following slightly larger but sparse linear system:
\begin{eqnarray}
\label{lineareqbig}
\cM
\begin{pmatrix}
d \\[5pt]
d_u
\end{pmatrix} =
\begin{pmatrix}
{\rm rhs} \\[5pt]
0
\end{pmatrix},\quad
\cM:=
\begin{pmatrix}
M_{\rm sp} & U       \\[5pt]
U^\top        & -D^{-1},
\end{pmatrix},\quad
d_u:=DU^\top d.
\end{eqnarray}
To obtain an accurate approximate solution to the linear system~\eqref{lineareqbig}, it is desirable to solve the above linear system via a preconditioned symmetric quasi-minimal residual method (PSQMR)~\cite{FreundPSQMR1994} with the preconditioner computed based on the following analytical expression of $\cM^{-1}$:
\begin{eqnarray*}
	\cM^{-1} =
	\begin{pmatrix}
		M_{\rm sp}^{-1} - M_{\rm sp}^{-1}US^{-1}U^\top M_{\rm sp}^{-1} & M_{\rm sp}	^{-1}US^{-1} \\[5pt]
		S^{-1}U^\top M_{\rm sp}^{-1} & -S^{-1}
	\end{pmatrix},
\end{eqnarray*}
where $S = D^{-1}+U^\top M_{\rm sp}^{-1}U$.
It can be readily seen that for a given vector $(h_1;h_2)$, $\cM^{-1}(h_1;h_2) $ can be evaluated efficiently as follows:
\[
\lambda_1  \;=\; M_{\rm sp}^{-1}h_1,\quad
\lambda_2  \;=\; S^{-1}\left(U^\top \lambda_1-h_2 \right),\quad
\cM^{-1}(h_1;h_2)  \;=\;
(\lambda_1-M_{\rm sp}^{-1}U\lambda_2; \lambda_2).
\]		
However, when the size of the matrix $S$ (which can be twice the number of second-order cones) is large or there is no obvious sparsity structure in the linear system~\eqref{lineareq}, the aforementioned technique may be time consuming. In this case, we would apply the PSQMR directly to solve the system~\eqref{lineareq} with diagonal preconditioner.

Next, we consider the case when $H\neq 0$. We then need to solve the following linear system as described in Algorithm iSSN,
\begin{equation}
	\label{eq-linsys-Qsocp}
	M\left(
		\begin{array}{c}
			d_1 \\[5pt]
			d_2
		\end{array}
	\right) =
	\left(
		\begin{array}{cc}
			H + \sigma HVH &  -\sigma HVA^\top \\[5pt]
			-\sigma AVH & \epsilon I_m + \sigma AVA^\top
		\end{array}
	\right)
	\left(
		\begin{array}{c}
			d_1 \\[5pt]
			d_2
		\end{array}
	\right) =
	\left(
		\begin{array}{c}
			HR_1 \\[5pt]
			R_2
		\end{array}
	\right)
\end{equation}
such that $(d_1,d_2)\in \mathrm{Ran}(H)\times \mathbb{R}^m$ and
\begin{equation}
	\label{eq-residual}
	\norm{M\left(
		\begin{array}{c}
			d_1 \\[5pt]
			d_2
		\end{array}
	\right) - \left(
		\begin{array}{c}
			HR_1 \\[5pt]
			R_2
		\end{array}
	\right)}\leq \nu,
\end{equation}
where $\epsilon>0$, $\sigma > 0$, and $\nu > 0$ are given parameters, $R_1$ and $ R_2$ are given vectors; and $V\in \partial_B \Pi_{\mathcal{K}}[-\sigma \tilde{y}(x_1, x_2,y)]$ at the given point $(x_1, x_2, y)$. Given the fact that one requires the condition $d_1\in \mathrm{Ran}(H)$ to establish the convergence of Algorithm iSSN, however, in practice this condition may bring numerical issues in computing the Newton direction. Fortunately,  we can fully overcome this difficulty via solving the following simplified system
	\begin{equation}
		\label{eq-linsys-unsym}
		\widehat{M}\left(
		\begin{array}{c}
			\hat{d}_1 \\[5pt]
			\hat{d}_2
		\end{array}
		\right) =
		\left(
		\begin{array}{cc}
			I_n + \sigma VH &  -\sigma VA^\top \\[5pt]
			-\sigma AVH & \epsilon I_m + \sigma AVA^\top
		\end{array}
		\right)
		\left(
		\begin{array}{c}
			\hat{d}_1 \\[5pt]
			\hat{d}_2
		\end{array}
		\right) =
		\left(
		\begin{array}{c}
			R_1 \\[5pt]
			R_2
		\end{array}
		\right)
	\end{equation}
	such that $(\hat{d}_1, \hat{d}_2)\in \mathbb{R}^n\times \mathbb{R}^m$, with the residual
	\[
		\norm{
			\widehat{M}
			\left(
				\begin{array}{c}
					\hat{d}_1 \\[5pt]
					\hat{d}_2
				\end{array}
			\right) -
			\left(
				\begin{array}{c}
					R_1 \\[5pt]
					R_2
				\end{array}
			\right)
		} \leq \frac{1}{\max\{1, \lambda_{\max}(H)\}} \nu,
	\]
	where $\lambda_{\max}(H)$ is the maximum eigenvalue of $H$. Then simple calculations show that $(d_1, d_2) := (\Pi_{\mathrm{Ran}(H)}(\hat{d}_1), \hat{d}_2)$ solves \eqref{eq-linsys-Qsocp} satisfying \eqref{eq-residual}. Moreover, one can verify that $H\Pi_{\mathrm{Ran}(H)}(\hat{d}_1) = H\hat{d}_1$ and $\inprod{\Pi_{\mathrm{Ran}(H)}(\hat{d}_1)}{H\Pi_{\mathrm{Ran}(H)}(\hat{d}_1)} = \inprod{\hat{d}_1}{H\hat{d}_1}$. Using these facts and analyzing the proposed algorithm carefully, we can execute the proposed algorithm without computing $\Pi_{\mathrm{Ran}(H)}(\hat{d}_1)$ explicitly. Finally, to solve the linear system \eqref{eq-linsys-unsym}, we can apply a direct method via computing the sparse LU factorization of $\widehat{M}$ if it is sparse. Otherwise, we may use an iterative solver, such as the BICGSTAB method  in \cite{SaadIterative2003}.

\subsection{Settings for numerical experiments}
\label{subsec-settings}
In this subsection, we present the settings of our numerical experiments. We first set up the stopping criteria for  the proposed ALM based on the KKT conditions given in~\eqref{kkt-condition}. We define the following relative KKT residuals,
\begin{eqnarray*}
	& &	\Delta_1(x_1,y) :=
	\frac{\sqrt{\norm{Ay-b}^2+\norm{H(x_1-y)}^2}}{1+\norm{b}+\norm{H}_F},\quad \Delta_2(y,x_3) :=
	\frac{\norm{x_3 - \Pi_{\cK}(x_3-y)}}{1+\norm{y}+\norm{x_3}}, \\[5pt]
	& & \Delta_3(x_1,x_2,x_3):=
	\frac{\norm{-Hx_1 + A^\top x_2+x_3-c}}{1+\norm{c}}
\end{eqnarray*}
and the relative gap
\begin{eqnarray*}
	\Delta_4(x_1,x_2,y):=
	\frac
	{
		|{\rm pobj} - {\rm dobj}|
	}
	{
		1 +
		|{\rm pobj}| +
		|{\rm dobj}|
	},
\end{eqnarray*}
where ${\rm pobj}:= \frac{1}{2}\inprod{x_1}{Hx_1} - \inprod{b}{x_2}$ and ${\rm dobj}:= -\frac{1}{2}\inprod{y}{Hy} - \inprod{c}{y}$ are the objective function values for primal and dual problems, respectively. For any given termination tolerance ${\rm tol}$, which will be specified later, we terminate our ALM solver when
\begin{eqnarray}
\label{relativekkt}
\Delta^k := \max
\left\{
\Delta_1(x_1^k,y^k),\;
\Delta_2(y^k,x_3^k),\;
\Delta_3(x_1^k, x_2^k,x_3^k),\;
\Delta_4(x_1^k,x_2^k,y^k)
\right\} <
{\rm tol},
\end{eqnarray}
where $\{(x_1^k,x_2^k,x_3^k,y^k)\}$ is the sequence generated by the algorithm at the $k$-th iteration.

In our numerical experiments, we will consider both linear and convex quadratic SOCPs. For linear SOCPs, the solvers that we will benchmark against are the highly powerful commercial solver Mosek\footnote{\url{https://www.mosek.com/}} (version 9.1.7) and the efficient open source semidefinite-quadratic-linear programs (SQLP) solver SDPT3\footnote{\url{https://blog.nus.edu.sg/mattohkc/softwares/sdpt3/}}~\cite{Tutuncu2003sdpt3} (version 4.0). For the convex quadratic SOCPs, we apply our ALM solver to the problem with quadratic objective directly, while for Mosek and SDPT3, we solve the reformulated problem~\eqref{eq-reformulatedsocp}.

For the ALM, we set $ {\rm tol} = 10^{-8} $ and stop the algorithm whenever it returns a solution such that $\Delta^k$ defined in~\eqref{relativekkt} is less than ${\rm tol}$. Moreover, the maximum number of iterations for the ALM is set to be 100. Since Mosek solves a homogeneous self-dual model which uses different stopping criteria, we use its default settings. The solutions returned by Mosek and SDPT3 under the default settings are then extracted to compute the relative KKT residuals in~\eqref{relativekkt}. We observe that when the default settings are used, Mosek and SDPT3 provide similar levels of accuracy as ours in terms of relative KKT residuals defined in~\eqref{relativekkt}.

All the computational results are presented in tables. The column under ``{\rm it}'' reports the number of iterations for each algorithm. Note that for the column ``{\rm it(newton)}'', we report the number of ALM iterations and the total number of Newton systems solved in the ALM. In addition,  the column ``{\rm time}'' reports the computational time in seconds. For the column ``{\rm kkt}'', we report the relative KKT residuals returned by each solver.

All experiments are run in MATLAB R2018b on a workstation with Intel Xeon processor E5-2680v3 at 2.50GHz (this processor has 12 cores and 24 threads) and 128GB of RAM, equipped with 64-bit Windows 10 operating system. Since Mosek can take advantage of multi-threading, we observe that under this operating system, the number of threads used by Mosek is 12, whereas for SDPT3 and our solver, only one thread is observed to be used by MATLAB.

\subsection{Application to minimal enclosing ball problems}
\label{subsec-meb}
In this subsection, we consider the MEB whose goal is to compute a ball of smallest radius that encloses a given set of balls (including points). The MEB problem is a member of the family of {\it minimum containment problems}, and it is also known as the {\it smallest enclosing ball problem} and {\it minimal bounding sphere problem}, etc. We refer the reader to~\cite{ZhouMEB2005} for an introduction of MEB problems.

Let $B_i$ denote a ball in $\R^d$ with center $c_i$ and radius $r_i\geq 0$, i.e.,
$$
B_i = \left\{ z\in \R^d\;:\; \|z-c_i\|\leq r_i \right\}.
$$
Given a set of distinct balls $\cB= \left\{ B_1,B_2,\cdots,B_m \right\}\subseteq \R^d$, the MEB problem is equivalent to the following unconstrained convex minimization problem:
\begin{eqnarray}
\label{meb-unconst}
\min_{z\in \R^d}\;\max_{1\leq i \leq m}\; \left\{ \|z-c_i\|+r_i \right\}.
\end{eqnarray}
Since the objective function is nonsmooth, the usual gradient-based methods are not applicable. However, if we denote $ n = m(d+1),\;\mathbf{x_2} = (r;z)\in \R^{d+1} $, and
$$
\mathbf{x_3} = (t_1;s_1;t_2;s_2;\cdots;t_m;s_m)\in \R^{n},
$$
problem (~\ref{meb-unconst}) can be reformulated into
a linear SOCP problem of  the form~\eqref{QLP} (see, e.g.,~\cite{ZhouMEB2005} for such a reformulation):
\begin{eqnarray}
\label{meb-dsocp}
({\rm MEB})\quad
\max_{\mathbf{x_2},\bm{x_3}} \;
\left\{
\left.
\mathbf{b}^\top \mathbf{x_2} \;\right\vert\; \mathbf{A}^\top \mathbf{x_2} + \mathbf{x_3} = \mathbf{c},\; \mathbf{x_3}\in \cK
\right\},
\end{eqnarray}
where
\begin{eqnarray*}
	& & \mathbf{b} = -(1;0;\cdots;0)\in \R^{d+1}, \quad
	\mathbf{c} = -(r_1;c_1;r_2;c_2;\cdots; r_m;c_m)\in \R^{n},\\[5pt]
	& &	\mathbf{A} = - \underbrace{
		\begin{pmatrix}
			I_{d+1}& \cdots & I_{d+1}
	\end{pmatrix}}_{m}\in \R^{(d+1)\times n},
	\cK = \underbrace{\cK^{d+1}\times \cdots\times \cK^{d+1}}_{m}\subseteq \R^n.
\end{eqnarray*}
Then we can apply the proposed ALM to solve the MEB problem. To achieve a fast local convergence rate for the semismooth Newton method when solving the ALM subproblems, we need the constraint nondegeneracy condition. For the MEB problem, by considering its geometrical properties, we are able to show that the constraint nondegeneracy condition holds at any feasible solution of the dual problem of MEB.

\begin{theorem}
	\label{thm-nondegmeb}
	Assume that $\cB= \left\{ B_1,B_2,\cdots,B_m \right\}\subset \R^d$ with  $m>1$. Then the constraint nondegeneracy condition holds at any feasible solution $\bar{\mathbf{y}}$ for the  dual problem of (MEB), i.e.,
	\begin{eqnarray*}
		\mathbf{A}(\text{lin}(\cT_{\cK}(\bar{\mathbf{y}}))) =
		\R^{d+1}\quad \forall\, \bm{A}\bar{\bm{y}} = \bm{b}.
	\end{eqnarray*}
\end{theorem}
\begin{proof}
	Let $\bar{\mathbf{y}} = (\bar{\mathbf{y}}_1; \ldots; \bar{\mathbf{y}}_m)\in \R^n$ be any feasible solution, i.e.,  $\bm{A}(\bar{\bm{y}}) = \bm{b}$ and $\bar{\mathbf{y}}_i = (\alpha_i; v_i) \in \cK^{d+1}$ for
	$i=1,\ldots,m.$
	If there exists $i$ such that $\alpha_i>\norm{v_i}$, then the conclusion is trivial since $\text{lin}(\cT_{\cK^{d+1}}(\bar{\mathbf{y}_i}))  = \R^{d+1}$. Assume without loss of generality that for all $1\leq i\leq m_0$, $m_0\leq m$, we have that $\alpha_i = \norm{v_i}>0$ and that for all $ i> m_0$, $\alpha_i = \norm{v_i}=0$.
	
	We claim that $m_0\geq 2$. If $m_0 = 1$, then by the feasibility condition, we have that $\alpha_1 = 1$, and $v_1 = 0$, but this is impossible since $\bar{\mathbf{y}}_1 = (\alpha_1;v_1)\in \partial \cK^{d+1}$. Thus $m_0\geq 2.$ Next, we show that there exist $1\leq i<j\leq m_0$ such that $\bar{\mathbf{y}}_i = (\alpha_i;v_i)$ and $\bar{\mathbf{y}}_j=(\alpha_j;v_j)$ is linearly independent. Suppose that this is not true. Then all the vectors $\bar{\mathbf{y}}_i$, $1\leq i\leq m_0$,  are parallel, and  the feasibility condition implies that $v_i = 0$ for all $1\leq i\leq m_0$. The latter contradicts the assumption that $\bar{\mathbf{y}}_i$, $1\leq i\leq m_0$, are nonzero vectors on the boundary of $\cK^{d+1}$. Now for such $i$ and $j$, the linearity spaces are given by
	\begin{eqnarray*}
		\text{lin}(\cT_{\cK^{d+1}}(\alpha_i;v_i)) =
		(\alpha_i;-v_i)^\perp,\quad \text{lin}(\cT_{\cK^{d+1}}(\alpha_j;v_j)) =
		(\alpha_j;-v_j)^\perp.
	\end{eqnarray*}
	For the primal constraint nondegeneracy condition to hold, we need to show that
	\begin{eqnarray*}
		\textrm{lin}(\cT_{\cK^{d+1}}(\alpha_i;v_i)) +
		\textrm{lin}(\cT_{\cK^{d+1}}(\alpha_j;v_j)) = \R^{d+1}.
	\end{eqnarray*}
	However, the aforementioned condition is equivalent to
	\begin{eqnarray*}
		\textrm{span}\left\{(\alpha_i;-v_i)\right\}\cap
		\textrm{span}\left\{(\alpha_j;-v_j)\right\} =
		\left\{ 0\right\},
	\end{eqnarray*}
	which holds true because of the linear independence of the vectors $\bar{\mathbf{y}}_i$ and $\bar{\mathbf{y}}_j$. This completes the proof.
\end{proof}

Next, we evaluate the performance of the proposed ALM against SDPT3 and Mosek. Let $\{\bar{p}_i\}_{i\geq 0}$ denote the following pseudo-random sequence:
\begin{eqnarray*}
	p_0 = 7,\quad
	p_{i+1} = (445p_i+1)\;{\bf mod}\;4096,\quad
	\bar{p}_i = \frac{p_i}{40.96},\quad
	i = 1,2,\cdots.
\end{eqnarray*}
Then the elements of $c_i$, $ i = 1,2,\cdots,m $ are successively set to $\bar{p}_1,\bar{p}_2,\cdots$, in the order
\begin{eqnarray*}
	r_1,(c_1)_1,\cdots,(c_1)_d,\cdots,r_m,(c_m)_1,\cdots, (c_m)_d.
\end{eqnarray*}
Note that same testing instances were also used in~\cite{ZhouMEB2005}. The associated computational results are presented in Table~\ref{tab-meb}. From these results, we observe that SDPT3, Mosek, and the ALM solve all the instances successfully. Our ALM outperforms the other methods in the sense that the computational time is  much smaller. Mosek outperforms SDPT3 but becomes less efficient when the problem size is large. Indeed, Mosek is about two times faster than SDPT3 while the ALM is at least two times faster than Mosek when the problem size is large. Thus we can conclude that the proposed ALM is highly efficient and robust for MEB problems.

\begin{table}[H]
	\begin{center}
		\begin{footnotesize}
			\caption{Computational results for MEB problems with various value of $m$ and $d$.}
			\label{tab-meb}
			\begin{tabular}{|c|r|r|r|} \hline
				& \mc{1}{c|}{SDPT3}  &  \mc{1}{c|}{Mosek} &  \mc{1}{c|}{ALM} \\ \hline
				m,\;d &  \mc{1}{c|}{it$|$time$|$kkt}	&  \mc{1}{c|}{it$|$time$|$kkt}
				&  \mc{1}{c|}{it(newton)$|$time$|$kkt} \\ \hline
				1000,   400 	& 21 $|$ 5.7 $|$ 9.6e-09 	&       13 $|$ 2.7 $|$ 3.5e-09 	&    7(40) $|$ 1.8 $|$ 2.9e-09\\
				1000,   800 	& 22 $|$ 13.1 $|$ 6.5e-09 	&      14 $|$ 5.2 $|$ 1.6e-09 	&    7(44) $|$ 3.0 $|$ 1.9e-09\\
				1000,  1200		& 21 $|$ 19.8 $|$ 8.0e-09 	&      12 $|$ 8.2 $|$ 1.2e-09 	&    7(42) $|$ 4.1 $|$ 1.3e-09\\
				1000,  1600 	& 21 $|$ 29.7 $|$ 7.3e-09 	&      11 $|$ 11.7 $|$ 7.3e-09 	&   6(39) $|$ 5.3 $|$ 3.2e-09 \\
				1000,  2000	 	& 19 $|$ 34.3 $|$ 8.1e-09 	&      13 $|$ 17.7 $|$ 4.3e-09	&   6(37) $|$ 6.0 $|$ 2.3e-09 \\
				8000,   100 	& 25 $|$ 13.4 $|$ 7.6e-09 	&      17 $|$ 5.2 $|$ 2.4e-09 	&    7(45) $|$ 3.6 $|$ 9.1e-09\\
				16000,   100 	& 25 $|$ 26.1 $|$ 7.1e-09 	&     19 $|$ 11.0 $|$ 2.2e-09 	&   8(48) $|$ 7.5 $|$ 4.1e-09 \\
				32000,   100 	& 25 $|$ 54.6 $|$ 8.4e-09 	&     20 $|$ 23.7 $|$ 1.7e-09 	&   7(44) $|$ 14.1 $|$ 2.2e-09 \\
				64000,   100 	& 28 $|$ 119.7 $|$ 6.7e-09 	&    20 $|$ 49.3 $|$ 4.9e-09 	&   7(45) $|$ 28.6 $|$ 1.9e-09 \\
				128000,   100 	& 30 $|$ 277.3 $|$ 5.8e-09 	&   18 $|$ 97.9 $|$ 6.4e-09 	&   6(43) $|$ 54.7 $|$ 3.9e-09 \\
				256000,   100 	& 30 $|$ 645.3 $|$ 1.0e-08 	&   20 $|$ 377.6 $|$ 1.2e-08 	&  6(45) $|$ 118.4 $|$ 3.6e-09  \\
				512000,   100 	& 31 $|$ 1429.2 $|$ 7.0e-09 &  20 $|$ 1306.5 $|$ 1.0e-08 	&  5(39) $|$ 212.4 $|$ 9.1e-09  \\
				3000,  1000 	& 21 $|$ 54.8 $|$ 7.9e-09 	&      14 $|$ 20.1 $|$ 4.4e-09	&    7(43) $|$ 10.5 $|$ 3.1e-09\\
				3000,  2000 	& 22 $|$ 123.1 $|$ 9.2e-09 	&     15 $|$ 51.9 $|$ 1.9e-10 	&   7(46) $|$ 21.5 $|$ 6.0e-09 \\
				3000,  4000 	& 22 $|$ 283.1 $|$ 5.9e-09 	&     11 $|$ 128.8 $|$ 1.5e-10 	&  6(40) $|$ 36.4 $|$ 4.0e-09  \\
				3000,  8000 	& 20 $|$ 558.6 $|$ 7.9e-09 	&     12 $|$ 277.2 $|$ 5.0e-09 	&  6(39) $|$ 71.2 $|$ 6.1e-09  \\
				3000, 16000 	& 20 $|$ 1334.9 $|$ 5.4e-09 &    12 $|$ 592.7 $|$ 2.3e-10 	&   6(44) $|$ 164.6 $|$ 1.6e-09 \\
				\hline
			\end{tabular}
		\end{footnotesize}		
	\end{center}
\end{table}

\subsection{Application to trust-region subproblems}
\label{subsec-trs}
We consider in this subsection SOCPs arising from the classical trust-region subproblem,
\begin{eqnarray}
\label{eq-classictrs}
\min_{y\in \R^d}\;
\left.
\left\{
\frac{1}{2}\inprod{y}{Hy} + \inprod{c}{y}\;\right\vert\; \norm{y}\leq 1
\right\},
\end{eqnarray}
where $H$ is symmetric but not necessarily positive semidefinite. It was proven in \cite[Theorem 5]{NamTRS2017} that when $\lambda_H<0$ (the smallest eigenvalue of $H$), a tight convex relaxation of the classical TRS~\eqref{eq-classictrs} can be derived and is given by
\begin{eqnarray}
\label{eq-convextrs}
\min_{y\in \R^d}\;
\left.
\left\{
\frac{1}{2}\inprod{y}{(H-\lambda_HI_d)y} + \inprod{c}{y} + \lambda_H\;\right\vert \; \norm{y}\leq 1
\right\}.
\end{eqnarray}
Problem~\eqref{eq-convextrs} can be reformulated (ignoring the constant term $\lambda_H$ in the objective) to the form of {\rm (D)},
\begin{eqnarray}
\label{eq-trssocp}
\min_{\bm{y}}\;
\left.
\left\{
\frac{1}{2}\inprod{\bm{y}}{\bm{H}\bm{y}} +
\inprod{\bm{c}}{\bm{y}} \;\right\vert \;\bm{A} \bm{y} = \bm{b},\; \bm{y}\in \cK^{d+1}
\right\},
\end{eqnarray}
where $\bm{y}:=(s,y)^\top\in \R^{d+1}$, $\bm{c} := (0,c)^\top\in \R^{d+1}$, $\bm{b} := 1$, $ \bm{A} := (1,0)\in \R^{1\times (d+1)}  $, and
\begin{eqnarray*}
	\bm{H} :=
	\left(
	\begin{array}{cc}
		0 & 0 \\[5pt]
		0 & H-\lambda_HI_d
	\end{array}
	\right)\in \S_+^{d+1}.
\end{eqnarray*}

To solve a problem of the form \eqref{eq-trssocp} by Mosek and SDPT3, we need also to reformulate it as a linear SOCP as we did in the introduction. Specifically, problem \eqref{eq-trssocp} is equivalent to the following problem with additional affine and rotated quadratic cone constraints:
\begin{eqnarray*}
	\min_{\bar{\bm{y}}}\;
	\left.
	\left\{
	\inprod{\bar{\bm{c}}}{\bm{\bar{y}}} \;\right\vert\; \bar{\bm{A}}\bar{\bm{y}} = \bar{\bm{b}},\; \bar{\bm{y}}\in \cK^{d+1}\times \cK_r^{d+2}
	\right\},
\end{eqnarray*}
where $\bar{\bm{y}}:=(\bm{y},t,q,z)^\top\in \R^{2d+3}$ with $t,q\in \R,\; z\in \R^d$, $\bar{\bm{c}}:=(\bm{c},1,0)^\top\in \R^{2d+3}$, $\bar{\bm{b}}:=(\bm{b},0,1)\in \R^{d+2}$ and
\begin{eqnarray*}
	\bar{\bm{A}}:=
	\left(
	\begin{array}{cccc}
		\bm{A} & 0 & 0 & 0        \\[5pt]
		R 	   & 0 & 0 & -I_{d+1} \\[5pt]
		0 & 0  & 1 & 0
	\end{array}
	\right)	\in \R^{(d+2)\times (2d+3)}
\end{eqnarray*}
with $\bm{H} = R^\top R$.

Next we compare the performance of our SOCP solver with Mosek on a class of synthetic data. In particular, we randomly generate the input date via the following MATLAB scripts:
\begin{verbatim}
    P = rand(d,d); h = (P*diag(randn(d,1)))*P';
    lamh = eigs(h,1,'smallestreal');
    H = [zeros(1,d+1);[zeros(d,1),h-lamh*eye(d)]];
    c = [0;randn(d,1)]; b = 1; A = [1,zeros(1,d)]; R = H^0.5;
\end{verbatim}
The computational results are presented in Table~\ref{tab-trs}. In the table, we also report the minimum eigenvalue of the data matrix $H$ (corresponding to $h$ in the above MATLAB script), which is denoted by the term $\lambda_H$. From the table, we can see that our ALM solver outperforms Mosek and SDPT3 in terms of computational time. In most cases, the solution quality returned by our solver is much better than that of Mosek and SDPT3. These results also indicate that dealing with the quadratic objective directly is indeed much more efficient.

\begin{table}[h]
	\begin{center}
		\begin{footnotesize}
			\caption{Computational results for trust-region subproblem on synthetic data.}
			\label{tab-trs}
			\begin{tabular}{|c|c|r|r|r|} \hline
				\multicolumn{2}{|c|}{} &\mc{1}{c|}{Mosek} & \mc{1}{c|}{SDPT3} & \mc{1}{c|}{ALM}
				\\ \hline
				n & $\lambda_H$ & \mc{1}{c|}{it$|$time$|$kkt} & \mc{1}{c|}{it$|$time$|$kkt}
				& \mc{1}{c|}{it(newton)$|$time$|$kkt}
				\\ \hline
				1000 & -9.9e+03 &  3 $|$ 0.9 $|$ 1.1e-09 &  9 $|$ 2.2 $|$ 1.1e-07 & 6(13) $|$ 0.1 $|$ 1.7e-09\\
				2000 & -1.2e+04 &  3 $|$ 2.4 $|$ 6.6e-08 &  8 $|$ 8.3 $|$ 2.6e-06 & 5(14) $|$ 0.1 $|$ 3.6e-09\\
				3000 & -1.3e+04 &  3 $|$ 6.0 $|$ 1.9e-06 &  9 $|$ 22.1 $|$ 9.9e-08 & 6(16) $|$ 0.3 $|$ 6.1e-10\\
				4000 & -1.5e+04 &  3 $|$ 11.9 $|$ 4.0e-09 &  9 $|$ 40.5 $|$ 1.1e-07 & 6(15) $|$ 0.6 $|$ 5.9e-11\\
				5000 & -1.8e+04 &  3 $|$ 18.8 $|$ 3.5e-09 &  9 $|$ 66.6 $|$ 1.2e-07 & 6(13) $|$ 0.8 $|$ 1.0e-09\\
				6000 & -4.1e+04 &  3 $|$ 29.4 $|$ 9.9e-08 &  9 $|$ 106.9 $|$ 1.2e-07 & 6(13) $|$ 1.0 $|$ 3.7e-10\\
				7000 & -5.3e+04 &  3 $|$ 40.7 $|$ 2.4e-07 &  9 $|$ 149.8 $|$ 1.6e-07 & 6(13) $|$ 1.4 $|$ 4.6e-11\\
				8000 & -6.6e+04 &  3 $|$ 56.6 $|$ 1.3e-06 &  9 $|$ 205.8 $|$ 1.0e-07 & 6(15) $|$ 2.0 $|$ 3.2e-11\\
				9000 & -1.4e+05 &  3 $|$ 72.5 $|$ 1.4e-07 &  9 $|$ 269.4 $|$ 1.8e-07 & 5(10) $|$ 1.7 $|$ 4.0e-09\\
				10000 & -1.2e+05 &  3 $|$ 88.1 $|$ 7.5e-10 &  9 $|$ 292.8 $|$ 5.5e-08 & 5(11) $|$ 2.1 $|$ 8.1e-10\\
				\hline
			\end{tabular}
		\end{footnotesize}
	\end{center}
\end{table}

\subsection{Application to square-root Lasso problems}
\label{subsec-srLasso}
In this experiment, we consider the following square-root Lasso model proposed in~\cite{Bellonisrlasso2011}:
\begin{eqnarray}
\label{eq-srlasso}
\min_{x\in \R^d}\; \norm{Bx-w}+\lambda \norm{x}_1,
\end{eqnarray}
$B\in \R^{m\times d}$  and $w\in \R^m$ are given data, $m$ is
the sample size, and $d$ is the dimension of the features.

As explained in~\cite{Bellonisrlasso2011}, the square-root Lasso model is  advantageous over the classical Lasso model. When dealing with noise that follows a Gaussian distribution $\cN(0,\sigma^2)$, the square-root Lasso model guarantees a near-oracle performance. Moreover, for the square-root Lasso model, one does not need to know an estimate of the standard deviation $\sigma$ in advance, while such an estimate of $\sigma$ is needed in the classical Lasso model. However, it is nontrivial to estimate the standard deviation when the dimension of features, $d$, is much larger than the sample size, $ m $. Therefore, the square-root Lasso model is in some sense more useful.

It is outside the scope of this paper to compare the empirical performance of different models from statistical perspective. Here  we focus on the numerical aspects of solving the optimization problem~\eqref{eq-srlasso} by reformulating it into an SOCP of the form~\eqref{QLP}. Hence, we only compare the performance of our proposed ALM against other general SOCP solvers but not the specialized square-root Lasso solvers such as the one in~\cite{Tangsqlasso2019}.

As stated in~\cite[Section 4]{Bellonisrlasso2011}, problem~\eqref{eq-srlasso} can be equivalently reformulated as a standard SOCP. Indeed, we note that for any real number $a$, we have $|a| = a_++a_{-}$ and $a = a_+-a_{-}$, where $a_+$ and $a_{-}$ denote the positive and negative parts of $a$, respectively. Therefore, we can write $x = p-q$, with $p,\,q\in \R_+^d$ and thus,
\begin{eqnarray*}
	\norm{Bx-w}+\lambda \norm{x}_1 =
	\norm{Bp-Bq-w}+\lambda e^\top (p+q),\quad e =
	\begin{pmatrix}
		1&\cdots&1
	\end{pmatrix}^\top \in\R^d.
\end{eqnarray*}
Now let $z = Bp-Bq-w$. Then~\eqref{eq-srlasso} is equivalent to
\begin{eqnarray}
\label{eq-srlassosocp}
\min_{(t,z),p,q}\;
\left.
\left\{
t+\lambda e^\top (p+q)\;\right\vert\; Bp-Bq-z = w,\;(t,z)\in \cK^{m+1},\;
p,\,q\in \R_+^d
\right\},
\end{eqnarray}
where $\cK^{m+1}$ is the second order cone in $\R^{m+1}$.
Denote $ \mathbf{y} = (p,q,t,z)^\top \in \R^{2d+m+1} $ and
\begin{eqnarray*}
	\bm{b}&:=&w\in\R^m,\quad
	\mathbf{c} := (\lambda e,\lambda e, 1, 0)^\top\in \R^{2d+m+1} ,\\
	\mathbf{A} &:=&
	\begin{pmatrix}
		B & -B & 0 & -I_m
	\end{pmatrix}\in \R^{m\times (2d+m+1)}.
\end{eqnarray*}
Then we obtain a standard SOCP in the form of the  dual problem of \eqref{QLP}:
\begin{eqnarray}
\label{srLasso-socp}
({\rm srLasso})\quad
\min_{\mathbf{y}} \;
\Big\{
\mathbf{c}^\top \mathbf{y}  \;\Big\vert\; \mathbf{A} \mathbf{y}  =
\mathbf{b},\; \mathbf{y} \in
\cK := \R_+^d\times \R_+^d\times \cK^{m+1}
\Big\}.
\end{eqnarray}

Next, we would test the reformulated problem \eqref{srLasso-socp} using SDPT3, Mosek, and our linear SOCP solver on a collection of UCI dataset\footnote{\url{https://archive.ics.uci.edu/}} which provides the data $B$ and $w$. For the choice of the regularization parameter, we follow the recent work of Tang et al.~\cite{Tangsqlasso2019}, where they adopted a 10-fold cross validation to estimate the best regularization parameter. In particular, we set the parameter $\lambda = c_0\Phi^{-1}(1-\frac{1}{40n})\lambda_c$, with $c_0 = 1.1$.

The choice of $\lambda_c$ and computational results are both presented in Table~\ref{tab-srLasso}. From the table, we observe that the three SOCP solvers can successfully solve all the instances. In terms of efficiency, we can see that the ALM has better performance than Mosek while SDPT3 is less efficient.

\footnotesize
\begin{table}[H]
	\caption{Computational results for square-root Lasso problems on UCI dataset.} \label{tab-srLasso}
	\centering
	\resizebox{0.95\textwidth}{!}{
		\begin{tabular}[!]{|c|c|c|r|r|r|}
			\hline
			\multicolumn{3}{|c|}{}	& \mc{1}{c|}{SDPT3}  & \mc{1}{c|}{Mosek} & \mc{1}{c|}{ALM}  \\ \hline
			problem & $\lambda_c$ & nnz &
			\mc{1}{c|}{it$|$time$|$kkt} & \mc{1}{c|}{it$|$time$|$kkt}
			& \mc{1}{c|}{it(newton)$|$time$|$kkt} \\ \hline
			$ \underset{(3308,150358)}{\rm E2006.test} $              & 0.107 & 1 & 13 $|$ 102.8 $|$ 8.3e-10 &14 $|$ 16.6 $|$ 3.1e-11 &4( 7) $|$ 3.8 $|$ 2.7e-09 \\
			$ \underset{  (74,201376)}{\rm pyrim.scaled.expanded5} $  & 0.619 &  48 & 42 $|$ 32.4 $|$ 2.6e-10 &26 $|$ 11.8 $|$ 1.9e-09 & 4(61) $|$ 10.2 $|$ 2.5e-09 \\
			$ \underset{(4177,6435)}{\rm abalone.scale.expanded7} $   & 0.020 & 32 & 22 $|$ 83.1 $|$ 2.2e-09 &14 $|$ 44.1 $|$ 6.3e-10 &12(37) $|$ 23.2 $|$ 5.0e-09 \\\
			$ \underset{ (252,116280)}{\rm bodyfat.scale.expanded7} $ & 0.067 &  15 & 34 $|$ 42.7 $|$ 3.4e-10 &20 $|$ 34.1 $|$ 8.5e-09 & 4(57) $|$ 5.9 $|$ 7.5e-09 \\
			$ \underset{ (506,77520)}{\rm housing.scale.expanded7} $  & 0.433 &  52 & 30 $|$ 54.9 $|$ 1.2e-09 &20 $|$ 47.2 $|$ 4.8e-09 & 8(52) $|$ 4.6 $|$ 2.1e-09 \\
			$ \underset{ (392,3432)}{\rm mpg.scale.expanded7} $       & 0.253 &  28 & 23 $|$ 2.0 $|$ 5.5e-10 &13 $|$ 1.3 $|$ 1.2e-09 &10(35) $|$ 0.7 $|$ 8.4e-09 \\
			$ \underset{(3107,5005)}{\rm space.ga.scale.expanded9} $  & 0.058 & 16 & 22 $|$ 42.5 $|$ 3.1e-09 &11 $|$ 22.0 $|$ 1.3e-08 & 6(27) $|$ 7.9 $|$ 2.5e-10 \\  \hline
		\end{tabular}
	}
\end{table}
\normalsize

\subsection{Numerical experiments on Dimacs Challenge problems}
\label{subsec-random}
In this subsection, we test each algorithm on the linear SOCPs in DIMACS Challenge\footnote{\url{http://archive.dimacs.rutgers.edu/Challenges/Seventh/Instances/}}. These  instances are commonly used to evaluate the efficiency and accuracy of linear SOCP solvers, and they are quite challenging to solve since many of the instances are highly degenerate.

The computational results are presented in Table~\ref{tab-dimacs}. From the table, we observe that the three methods are able to solve all the instances to the desirable accuracy except for the last few instances. For the computational time, we see that ALM takes a longer time  than Mosek and SDPT3 for solving many of the instances, especially the last few instances for which the ALM takes over 1000 Newton iterations to converge. Those instances, as far as we know, are highly degenerate, and it is the degeneracy that causes the slow convergence of the semismooth Newton method. This observation indicates that the ALM may perform poorly on degenerate problems. Finally, based on the presented numerical results, SDPT3 is also observed to be a highly efficient and robust solver for the {DIMACS} Challenge problems.

\begin{table}[h]
	\begin{center}
		\begin{footnotesize}
			\caption{Computational results on DIMACS Challenge problems.}
			\label{tab-dimacs}{
				\begin{tabular}{|c|r|r|r|} \hline
					& \mc{1}{c|}{SDPT3} & \mc{1}{c|}{Mosek} & \mc{1}{c|}{ALM} \\ \hline
					Problem & \mc{1}{c|}{it$|$time$|$kkt} & \mc{1}{c|}{it$|$time$|$kkt} & \mc{1}{c|}{it(newton)$|$time$|$kkt} \\ \hline
					nb & 22 $|$ 0.4 $|$ 3.1e-09 & 10 $|$ 0.4 $|$ 8.1e-09 & 11(46) $|$ 1.0 $|$ 2.8e-12\\
					nbL1 & 30 $|$ 4.0 $|$ 1.7e-09 & 12 $|$ 0.3 $|$ 1.3e-09 & 22(62) $|$ 4.9 $|$ 1.5e-12\\
					nbL2bessel & 20 $|$ 0.4 $|$ 9.6e-10 &  8 $|$ 0.2 $|$ 2.4e-13 &  8(15) $|$ 0.3 $|$ 1.9e-09\\
					nbL2 & 15 $|$ 0.3 $|$ 3.1e-09 &  8 $|$ 0.3 $|$ 2.2e-10 & 11(49) $|$ 1.0 $|$ 3.8e-10\\
					nql30new & 26 $|$ 1.0 $|$ 4.3e-10 & 16 $|$ 0.3 $|$ 1.8e-10 & 38(104) $|$ 1.7 $|$ 9.6e-09\\
					nql60new & 27 $|$ 4.5 $|$ 1.9e-10 & 17 $|$ 1.0 $|$ 9.2e-11 & 38(109) $|$ 7.7 $|$ 7.9e-09\\
					nql180new & 33 $|$ 40.4 $|$ 6.1e-11 & 21 $|$ 10.4 $|$ 6.8e-10 & 33(126) $|$ 95.5 $|$ 8.9e-09\\
					qssp30new & 20 $|$ 0.7 $|$ 3.9e-10 & 13 $|$ 0.3 $|$ 5.5e-11 & 13(42) $|$ 1.0 $|$ 7.7e-09\\
					qssp60new & 23 $|$ 3.6 $|$ 4.5e-10 & 13 $|$ 0.8 $|$ 1.9e-10 & 21(60) $|$ 7.0 $|$ 6.9e-11\\
					qssp180new & 29 $|$ 54.5 $|$ 8.1e-10 & 19 $|$ 12.8 $|$ 9.3e-10 & 21(69) $|$ 88.1 $|$ 9.6e-09\\
					sched5050s & 27 $|$ 0.9 $|$ 1.4e-09 & 21 $|$ 0.3 $|$ 2.6e-08 & 13(44) $|$ 0.5 $|$ 9.2e-09\\
					sched10050s & 29 $|$ 1.7 $|$ 7.7e-08 & 20 $|$ 0.4 $|$ 5.0e-07 & 67(1381) $|$ 30.1 $|$ 2.5e-09\\
					sched100100s & 28 $|$ 4.0 $|$ 9.1e-09 & 23 $|$ 0.8 $|$ 2.4e-06 & 100(1597) $|$ 39.3 $|$ 1.7e-06\\
					sched200100s & 36 $|$ 12.7 $|$ 1.0e-07 & 24 $|$ 1.5 $|$ 5.3e-08 & 52(1119) $|$ 59.5 $|$ 7.5e-09\\
					\hline
			\end{tabular} }
		\end{footnotesize}
	\end{center}
\end{table}

\section{Concluding remarks}
\label{sec-conclusion}
In this paper, we have employed the inexact ALM to solve  convex quadratic second-order cone programming problems (SOCPs). Under the quadratic growth condition, the KKT residual is shown to  possess a R-superlinear convergence rate based on recently developed results in the related topics. We also provide sufficient conditions for the quadratic growth condition to hold. Numerically, a practical SOCP solver is designed and implemented based on the proposed semismooth Newton-based ALM. Extensive numerical results on solving various classes of SOCPs demonstrate that our solver is highly efficient and robust. It has comparable performance to the highly powerful commercial solver Mosek and outperforms the well-known open source semidefinite-quadratic-linear programming solver SDPT3 on the tested problems. With fruitful applications of SOCPs in many fields, we believe that our  solver could serve as a promising toolbox for solving large-scale SOCPs in real-world applications.

\end{document}